\documentclass{amsart}
\usepackage[mathscr]{eucal}
\usepackage{amsmath,amssymb}
\usepackage{graphics}
\usepackage{multirow}
\usepackage{hyperref}
\usepackage{color}

\newtheorem{thm}{Theorem}[section]
\newtheorem{THM}{Theorem}

\newtheorem{cor}[thm]{Corollary}
\newtheorem{prop}[thm]{Proposition}
\newtheorem{lemma}[thm]{Lemma}

\theoremstyle{definition}

\newtheorem{remark}[thm]{Remark}

\DeclareMathOperator{\Hom}{Hom}
\DeclareMathOperator{\Pic}{Pic}

\DeclareMathOperator{\Aut}{Aut}

\DeclareMathOperator{\sing}{sing}

\DeclareMathOperator{\Mor}{Mor}

\DeclareMathOperator{\codim}{codim}

\DeclareMathOperator{\ord}{ord}

\DeclareMathOperator{\alb}{alb}

\def\C{\mathbb C}
\def\P{\mathbb P}
\def\p{\mathfrak p}
\def\Z{\mathbb Z}

\def\F{\mathcal F}
\def\D{\mathcal D}
\def\G{\mathcal G}
\def\H{\mathcal H}

\def\KH{{K_{\H}}}
\def\KF{{K_{\F}}}
\def\KG{{K_{\G}}}
\def\KX{{K_X}}
\def\KY{{K_Y}}
\def\KZ{{K_Z}}

\input xy
\xyoption{all}

\begin{document}

\title[Foliations with trivial canonical class]
{Singular foliations  with trivial canonical class}
\author[F. Loray, J.V. Pereira and F. Touzet ]
{Frank LORAY$^1$,  Jorge Vit\'{o}rio PEREIRA$^{2}$ and Fr\'ed\'eric TOUZET$^1$}
\address{\newline $1$  Univ Rennes, CNRS, IRMAR - UMR 6625, F-35000 Rennes, France\hfill\break
$2$ IMPA, Estrada Dona Castorina, 110, Horto, Rio de Janeiro,
Brasil} \email{$^1$ frank.loray@univ-rennes1.fr, frederic.touzet@univ-rennes1.fr}
\email{$^2$ jvp@impa.br}
\subjclass{} \keywords{Foliation, Transverse Structure, Birational Geometry}

\begin{abstract}
This paper  describes the structure of singular codimension one foliations with  numerically trivial canonical bundle on complex projective manifolds.
\end{abstract}

\maketitle

\setcounter{tocdepth}{1}
\sloppy
\tableofcontents

\section{Introduction}
Let $\mathcal F$ be a singular holomorphic foliation on a compact complex  manifold $X$, and let $\KF$ be its canonical bundle.
In  analogy with the case of complex manifolds, the canonical  bundle of $\F$  is the line bundle on $X$
which, away from the singular set of $\F$, coincides with the bundle of differential forms of maximal degree along the leaves of $\F$.

As in the case of manifolds, one expects that $\KF$ governs much of the geometry of $\F$. When $X$ is a projective surface, this vague
expectation has already been turned into precise results. There is now  a birational classification of foliations on projective surfaces,
very much in the spirit of Enriques-Kodaira classification of projective surfaces, in terms of numerical properties of $\KF$, see \cite{McQ,Brunella}.

In this paper, we investigate the structure of singular holomorphic codimension one foliations on projective manifolds with $\KF$ numerically equivalent to zero.
We were dragged into the subject by a desire to better understand  previous results, most notably \cite{CerveauLinsNetoAnnals} and  \cite{Touzet},
which  we  recall below. Further motivation  comes from the study of holomorphic Poisson manifolds, see  \cite{Polishchuk, LimaPereira}.

\subsection{Previous results}\label{S:oldresults}
Cerveau and Lins Neto proved that the space of foliations
on $\P^3$ with $\KF =0$ has six irreducible components \cite{CerveauLinsNetoAnnals}, and gave a rather precise description
of them.
The paper \cite{Polishchuk} by Polishchuk contains  a classification
of Poisson structures on $\P^3$ under restrictive hypothesis on their singular set. But (non-zero) Poisson structures
without singular divisors on $3$-folds are nothing more than foliations with trivial canonical bundle, thus Polishchuk's result is a particular case of
Cerveau-Lins Neto classification.  Poisson structures on projective $3$-folds with isolated singularities (i.e. codimension one foliations with trivial canonical bundle)
are classified by Druel in \cite{Druel}.

Smooth codimension one foliations  with numerically trivial canonical bundle on compact K\"{a}hler manifold  $X$ are classified in \cite{Touzet}.
If $\F$ is one such foliation, then it fits into at  least one of the following descriptions.
\begin{enumerate}
\item The foliation $\F$ is an isotrivial fibration by hypersurfaces with zero first Chern class.
\item After a finite \'{e}tale covering,   $X$  is  product of a compact K\"ahler manifold  $Y$ with $c_1(Y)=0$ and a complex torus $T$ and $\F$ is the pull-back under the  natural
projection to $T$ of a linear codimension one foliation on $T$.
\item The manifold $X$ is a fibration by rational curves over a compact K\"{a}hler manifold $Y$ with $c_1(Y) =0$, and $\F$ is a foliation everywhere transverse to the
fibers of the fibration.
\end{enumerate}

\subsection{Main results}

One of  the first examples of foliations with $\KF=0$ that come to mind are those with
trivial tangent bundle. Foliations with trivial tangent bundle are
induced by  (analytic) actions of complex Lie groups which are locally free outside an
analytic subset of codimension at least two. If the action is not locally free, then a well-known result by Rosenlicht implies that the manifold  must be uniruled. We are able to generalize this
well-known fact, confirming a recent conjecture of Peternell \cite[Conjecture 4.23]{Peternell}.

\begin{THM}\label{TI:D}
Let  $X$ be a projective manifold and $L$ be a pseudo-effective line bundle on $X$.
If there exists  $v \in H^0(X,\bigwedge^p TX\otimes L^*)$ not identically zero but vanishing at some point, then $X$ is uniruled. In particular, if there exists a foliation $\F$ on $X$ with $c_1(T\F)$ pseudo-effective and $\sing(\F) \neq \emptyset$,
then $X$ is uniruled.
\end{THM}

Theorem \ref{TI:D}  reduces the task of classifying codimension one foliations
with $c_1(\KF)=0$ on arbitrary projective manifolds to  uniruled manifolds, as smooth foliations with numerically trivial canonical bundle  have already been classified in \cite{Touzet}.

Our main result provides a description of an arbitrary codimension one foliation
with trivial canonical class on an arbitrary projective manifold.
It should be compared with McQuillan's classification of foliations  of Kodaira
dimension zero on surfaces, \cite[IV.3]{McQ}.

\begin{THM}\label{TI:X}
Let $\F$ be a codimension one foliation with numerically trivial canonical bundle on a  projective manifold $X$
of dimension at least two. If the singularities of  $\F$  are not canonical, then $\F$ is uniruled.
Otherwise, if the singularities of $\F$ are canonical, then, perhaps after passing to an \'etale covering,
$X$ is a product of
a projective manifold $Y$ with trivial canonical bundle and a projective manifold $Z$; and   the foliation
$\F$ is the pull-back of a foliation $\G$ on $Z$ with trivial tangent sheaf.  Furthermore, if the general leaf of $\F$ is not algebraic, then  $\dim Z\ge 2$  and
$Z$ is a projective equivariant compactification of a complex  abelian Lie group and $\G$ is induced by the action of a codimension
one subgroup.
\end{THM}

It is not excluded from the statement above the case where $Y$ is zero-dimensional.
Notice also that, if $\dim Z=1$, then
$\G$ is a foliation by points and, consequently, $\F$ is a smooth foliation by algebraic leaves
with trivial canonical class.

\subsection{Outline}

Below, we describe the general structure of the proof of Theorem \ref{TI:X} and at the same time an outline of the paper.

In Section \ref{S:prelim}, we settle the basic terminology and recall some uniruledness criteria for  foliations. In Section \ref{S:canonical},
we discuss canonical singularities of codimension one foliations, and we show how the presence of non-canonical singularities allow
us to apply the uniruledness criteria previously discussed. This section also contains information on the polar divisor of closed rational $1$-forms defining foliations with canonical singularities. Section \ref{S:conormal} studies the cohomology of the conormal bundle
of foliations with trivial canonical class. The results there contained form the backbone of our strategy to prove Theorem \ref{TI:X}.
In particular, there one can find a proof that either $H^1(X,N^*\F)=0$, or $\F$ is uniruled, or $K_X$ is pseudo-effective. This leads us
to Section \ref{S:uniruled} where we prove Theorem \ref{TI:D} and consequently reduce our study to the category of uniruled manifolds.
Section \ref{S:free morphisms} studies deformation of free rational curves along codimension one foliations with numerically trivial canonical
class; there we show how these deformations lead either to the uniruledness of the foliation, or to the existence of a  transversely projective structure for it.
Section \ref{S:modp} studies the reduction of foliations with numerically trivial canonical bundle to fields of positive characteristic.
The outcome is  that foliations not defined by closed rational $1$-forms  have very well behaved singularities (outside a codimension
three subset they admit local holomorphic first integral). In Section \ref{S:structure}  it
is shown that  the existence of a transversely projective structure together with the constraints on the singularities obtained through reduction to positive characteristic implies the non vanishing of $H^1(X, N^* \mathcal F)$.  Finally, in Section \ref{S:structure2}, we study codimension one foliations with trivial canonical bundle defined by closed rational $1$-forms and conclude the proof of Theorem \ref{TI:X}.

\subsection{Acknowledgments} We are very grateful to St\'ephane Druel for useful discussions and for  bringing  \cite[Theorem VI.1.3]{kollar}  to our knowledge.
This paper also owns a lot to Michael McQuillan who  caught a number of mistakes in previous versions,
called our attention to the relevance of foliated canonical singularities to our study, and made a number
of other thoughtful suggestions.


\section{Preliminaries}\label{S:prelim}

\subsection{Foliations}
A foliation $\F$ on a complex manifold $X$ is determined by a coherent subsheaf $T\F$
of the tangent sheaf $TX$ of $X$ such that
\begin{enumerate}
	\item $T\F$ is  closed under the Lie bracket (involutive), and that
	\item the inclusion $T\F \to TX$ has torsion free cokernel.
\end{enumerate}
The locus of points where $TX / T\F$ is not locally free is
called the singular locus of $\F$, denoted here by $\sing(\F)$.

Condition (1) allows us to apply Frobenius Theorem  to ensure that for  every point $x$  in the complement of $\sing(\F)$, the
germ of  $T \F$ at $x$ can be identified with relative tangent bundle of a germ of smooth fibration $f : (X,x) \to (\mathbb C^q,0)$.
The integer $q = q(\F)$  is the codimension of $\F$. Condition (2) is of different nature and is imposed to avoid the existence of {\it removable} singularities. In particular, it implies
that the codimension of $\sing(\F)$ is at least two.

The dual of $T\F$ is the cotangent sheaf of $\F$ and will be denoted by $T^*\F$. The determinant of $T^*\F$, i.e. $(\wedge^{p} T^* \F)^{**}$ where $\dim(X)=n=p+q$,
is  the canonical bundle of $\F$ and will be denoted by $\KF$.

There is a dual point of view where $\F$ is determined by a subsheaf $N^* \F$ of the cotangent sheaf $\Omega^1_X = T^* X$ of $X$. The involutiveness
asked for in condition (1) above is replace by integrability: if $d$ stands for the exterior derivative, then
$dN^*\F \subset N^* \F \wedge \Omega^1_X$ at the level of local sections.
Condition (2) is unchanged: $\Omega^1_X / N^* \F$ is torsion free.

The normal bundle of $\F$ is defined as the dual of $N^* \F$. Over the smooth locus $X - \sing(\F)$ we have the following exact sequence
\[
0 \to T\F \to TX \to N \F \to 0 \, ,
\]
but  this is no longer exact over the singular locus. Anyway, as the singular set has codimension at least two
we obtain the adjunction formula
\[
\KX = \KF \otimes \det N^* \F
\]
valid in the Picard group of $X$.

The definitions above apply verbatim to foliations on smooth algebraic varieties defined
over an arbitrary field. But be aware that the geometric interpretation given by Frobenius Theorem will no longer hold, especially over fields
of positive characteristic.

\subsection{Rationally connected and uniruled foliations}
The   result stated below
is a particular case of a more general result by
Bogomolov and McQuillan proved in \cite{BogMac}, see also \cite{KST}. It generalizes
a Theorem of  Miyaoka, cf. \cite[Theorem 8.5]{Miyaoka}, \cite[Chapter 9]{kol}.

\begin{thm}\label{T:MBM}
	Let $\F$ be  foliation on a  complex projective manifold $X$. If there exists a curve $C \subset X$ disjoint from the
	singular set of $\F$ for which  $T\F _{|C}$ is ample,    then
	the leaves of $\F$ intersecting $C$ are algebraic, and the closure of a leaf of $\F$ through a general point of $C$ is a rationally connected variety.
\end{thm}

We recall that a variety $Y$ is  rationally connected  if, through any two points $x,y \in Y$, there exists a rational curve $C$ in $Y$ containing $x$ and $y$. Foliations
with all leaves algebraic and with rationally connected general leaf will be called rationally connected foliations.
From Theorem \ref{T:MBM}, one can easily deduce the following result closer to Miyaoka's original statement.

\begin{cor}\label{C:MBM}
Let $\F$ be a  foliation on a $n$-dimensional projective manifold $X$. If $T\F$ is semi-stable with respect to a polarization $H$, and  $\KF \cdot H^{n-1} < 0$, then
$\F$ is a rationally connected foliation.
\end{cor}
\begin{proof}
    If $m \gg 0$ and $C$ is a very general curve defined as a complete intersection of  elements of $|m H|$, then  $T\F_{|C}$ is
    a semi-stable vector bundle of positive degree according to \cite[Theorem 6.1]{MR}. Therefore, every quotient bundle of $T\F_{|C}$ has positive degree,
    and we can apply  \cite[Theorem 2.4]{Hart} to see that $T\F_{|C}$ is ample. We apply Theorem \ref{T:MBM} to conclude.
\end{proof}

Quite recently, Campana and Pa\u{u}n obtained an alternative  version of the above Corollary, see \cite{CPaun}.

\begin{thm}\label{T:MCP}
	Let $\F$ be a foliation on a projective manifold $X$. If $\KF$ is not pseudo-effective, then
	$\F$ is a uniruled foliation.
\end{thm}

While the conclusion is   weaker -- a foliation is uniruled if through a general point of the ambient space
passes a rational curve everywhere tangent to the foliation --,  the hypothesis is not only weaker, but also
considerably easier to check. This result already appeared implicitly in the proof of   Theorem 1.4 of \cite{CPT} but with a gap.

\subsection{Tangent subvarieties and pull-backs}

Let $\F$ be a singular foliation on a
projective manifold  $X$ of dimension $n$.
We will say that $\F$ is the pull-back of a foliation
$\mathcal G$ defined on a lower dimensional variety $Y$, say of dimension $k<n$,
if there exists a dominant rational map $\pi : X \dashrightarrow Y$ such that
$\F=\pi^*\mathcal G$. In this case, the leaves of $\F$
are covered by algebraic subvarieties of dimension $n-k$, the fibers of $\pi$.

Actually, the converse holds true. Suppose that, through a general point of $X$, there
exists an algebraic subvariety  tangent to $\F$ having codimension $k<n$. Since tangency
to $\F$ imposes a closed condition on the Hilbert scheme and $\mathbb C$ is uncountable, it follows that  the leaves of $\F$
are covered by $q$-dimensional algebraic subvarieties, $q\ge n-k$. More precisely, there exists an irreducible
algebraic variety $Y$ and an irreducible subvariety $Z\subset X\times Y$
such that the natural projections
\[
\xymatrix{
	Z \ar[r]^{\pi_2} \ar[d]^{\pi_1}  &Y  \\
	X  &}
\]
are both dominants; the   general fiber of $\pi_2$ has dimension $q$;
and the general fiber of $\pi_2$ projects to $X$ as a subvariety tangent to $\F$.
By   Stein factorization theorem, we can
moreover assume that   $\pi_2$ has irreducible general
fiber.

The following result shows how the existence of algebraic subvarieties through a general point and contained along the leaf through that point allows a factorization of the foliation. A particular version of
it can be found in the proof of \cite[Theorem 9.0.3]{kol}.

\begin{lemma}
    Let $\F$ be a  foliation on a projective
	manifold $X$ of dimension $n$. Assume that $\F$ is covered by
	a family of $(n-k)$-dimensional algebraic subvarieties as above.
	Then $\F$
	is the pull-back of a foliation defined on a variety $Y$ having dimension $\le k$.
\end{lemma}
\begin{proof}
	When $\pi_1:Z\to X$ is birational, which means that through a general point passes
	exactly one subvariety $Z_y$ of the family, then $\pi_2\circ\pi_1^{-1}:X\dashrightarrow Y$
	is the pull-back map.
	
	Suppose that our covering family $Z\subset X \times Y$  is such that $\dim (\pi_2^{-1}(y)) = q$ is  maximal.
	If $\pi_1$ is not birational, then take a general point $x \in X$ and let $y \in Y$  be such that
	$x \in Z_y= \pi_1  \pi_2^{-1}(y)$. Then
	$\pi_1 \pi_2^{-1} \pi_2 \pi_1^{-1} (Z_y)$ has dimension at least $q + 1$  at $x$ and is tangent to
	$\F$  by construction, see  \cite[Lemma 9.1.6.1]{kol}. This contradicts the maximality of the dimension. The lemma follows.
\end{proof}


\section{Canonical singularities}\label{S:canonical}

Let $\F$ be a foliation on a projective  manifold $X$. Following \cite[Section 2]{McQ}, see also \cite[Section 1]{MP}, we will say that $\F$ has canonical singularities if, for every birational morphism $\pi: Y \to X$ from a smooth projective manifold $Y$ to $X$, we have that the divisor $K_{\pi^* \F} - \pi^* \KF$ is effective. An irreducible subvariety $S \subset \sing(\F)$ is a non-canonical singularity for $\F$ if there exists a birational morphism $\pi: Y \to X$ and an irreducible component $E$ of the exceptional divisor of $\pi$ such that $\pi(E) = S$ and $\ord_E(K_{\pi^* \F} - \pi^* \KF)<0$. Otherwise $S$ is a canonical singularity of $\F$.

\begin{remark}
	If we blow-up a foliation with canonical singularities, then the resulting foliation may not have canonical singularities,
    even when we blow-up along smooth centers. The property of having canonical singularities is only preserved by blow-ups along
    centers invariant by the foliation or everywhere transverse to it, cf. \cite[Fact I.2.8]{McQ}.
\end{remark}

\begin{remark}
    In the definition of canonical/non-canonical singularities above, we can always restrict to birational morphisms which are obtained by composition of
    blow-ups along smooth centers. Indeed, by \cite[Theorem 7.17]{Hart} any birational morphism $\pi:Y\to X$ between projective varieties can be described
    as the blow-up along a suitable ideal $\mathcal I$. Hence, Hironaka's resolution of singularities implies that $\pi$ is dominated by
    $\tilde \pi : Z \to X$ a blow-up along smooth centers, i.e. there exists a birational morphism $q:Z \to Y$ such that $\tilde \pi = q \circ \pi$, cf. \cite[Theorem 2.5]{BEV}.
    Since $E$ has codimension one, $q$ is an isomorphism at a general point of $E$ and therefore the order of
    $K_{\pi^* \F} - \pi^* \KF$ along $E$ will be the same as the order of $K_{\tilde \pi^* \F} - \tilde \pi^* \KF$ along the unique irreducible divisor on $Z$ dominating $E$.
\end{remark}

For a thorough discussion of the concept of canonical singularities, and  other related kind of singularities
(terminal, log-canonical, log-terminal) the reader is redirected to \cite[Section I]{McQ}
where he can  find a  proof of the next proposition  \cite[Fact I.2.4]{McQ}.

\begin{prop}\label{P:canonicaldim2}
A singular point $p$ of a foliation $\F$ on a smooth surface $S$ is canonical if, and only if, there exists
a germ of holomorphic vector field $v$ at  $p$ defining   $\F$  having one
of the following forms:
\begin{enumerate}
    \item the vector field $v$ has semi-simple and invertible linear part with quotient of eigenvalues not belonging to $\mathbb Q^+$; or
    \item the vector field $v$ is analytically conjugated to $ x \frac{\partial}{\partial x} + (ny + x^{n+1}) \frac{\partial}{\partial y}$ for some positive integer $n$; or
    \item the vector field $v$ has linear part with zero determinant and non-zero trace.
\end{enumerate}
\end{prop}
\begin{remark}\label{R:dim2can}
Let us make some further comments on this result. The singularities of type (1) and (3) are the so-called reduced singularities and are the simplest models of singularities appearing after a suitable sequence of blowing-ups $\pi$ over any singular point of a foliation on a surface. This is Seidenberg's resolution Theorem (see \cite[Theorem 1, p.13]{Brunella}). The additional type (2) is called "Poincar\'e-Dulac" singularity (see \cite [p.114]{Brunella}). For instance, the type (1) (non degenerate reduced singularities) corresponds to foliations defined locally by a vector field of the form $v=\lambda_1 x\frac{\partial}{\partial x}+\ \lambda_2 y\frac{\partial}{\partial y}+ (\textit{higher order terms})$ with $\dfrac{\lambda_1}{\lambda_2}\notin{\mathbb Q}^+.$ Take a foliation $\F$ on a surface $S$ admitting only non degenerate reduced  singularities and consider the blow-up $\pi_p$ over a point $p$ in $S$. It can be easily verified that the exceptional divisor $E={\pi_p}^{-1}(0)$ is invariant by $\pi_p^*\F$ an that this latter foliation admits exactly one or two singularities on $E$, depending on whether $p$ is a regular point of $\F$ or not. Moreover, these singularities are also of type (1).
One can also check that the canonical bundle of $\F$ and $\pi_p^*\F$ are related by the formula
$$K_{\pi_p^*\F}=\pi_p^* K_\F +l(p)E$$
where $l(p)=1$ if $p$ is regular point of $\F$ and $l(p)=0$ otherwise. As a by-product, when $\pi$ is a sequence of blowing-ups, we get that $K_{\pi^*\F}={\pi}^* K_{\F} + D$ where $D$ is an effective divisor supported on the exceptionnal divisor of $\pi$. This gives an illustration of  Proposition \ref{P:canonicaldim2}.

It is also worth recalling that the union of germs of invariant analytic curves (separatrices) for this three types of singularies is a normal crossing curve (\cite[Chapter 1]{Brunella}). We refer again to \cite [Chapter 1,2]{Brunella} for further details on these aspects, including also singularities of type (2) and (3).
\end{remark}

It is not hard to deduce from Proposition \ref{P:canonicaldim2}  a description of    codimension two canonical singularities of codimension one foliations.

\begin{prop} \label{P:codim2}
Let $\F$ be a codimension one foliation $\F$ on a projective manifold. Let $S\subset \sing(\F)$ be a codimension two irreducible component of the singular set of $\F$.
If $S$ is a canonical singularity of $\F$, then, at a Euclidean neighborhood $U$ of a general point
of $S$, the foliation $\F$ is the pull-back under a submersion $f: U \to V$ of a foliation $\mathcal G$ with canonical singularities on a neighborhood $V$ of the origin of $\mathbb C^2$.
\end{prop}
\begin{proof}
    Let $p \in S$ be a general point and $\omega$ be a holomorphic $1$-form defining $\F$ at a neighborhood $U$ of $p$.
    After clearing denominators we can assume that $\omega$ has singular set of codimension at least two.

    Take a general surface $\Sigma \subset X$ intersecting $S$ transversally at $p$. Since $\dim \Sigma=2$
    all the singularities of $\mathcal G = \F_{|\Sigma}$ are isolated. We claim that $S$ is a canonical singularity for $\F$
    if, and only if, $p$ is canonical singularity for $\F_{|\Sigma}$.

    Indeed, let $\pi:Y \to X$ be a composition of blow-ups along smooth centers  and let $E$ be  an irreducible component
    of the exceptional divisor of $\pi$ such that $\pi(E) =S$.
    Since $S$ has codimension two, we may restrict $\pi$ to an open subset $Y_0$ of $Y$ which contains $\pi^{-1}(p)$ and such that the induced map
    $\pi_0 : Y_0 \to X_0$  is a composition of blow-ups along smooth centers of codimension two, all of them   dominating $S$.
    Since $\Sigma$ is transverse to $S$, all the successive strict transforms of $\Sigma$ are also transverse to the irreducible components of the singular set
    of the corresponding strict transforms of $\F$ which dominate $S$. Therefore
    $\tilde \Sigma$, the strict transform of $\Sigma$ under $\pi$, is smooth along $\pi^{-1}(p)$. Let $E_{\Sigma}$ be an irreducible component
    of $\tilde \Sigma$. Once one remarks that
    \[
        \ord_E (K_{\pi^*\F} - \pi^* \KF) = 1-  \ord_E( (\pi^*(\omega)_0) ).
    \]
    and that a similar formula holds for $K_{\mathcal G}$, it follows that
    \[
        \ord_E( K_{\pi^*\F} - \pi^* \KF) = \ord_{E_{\Sigma}} ( K_{(\pi_{|\tilde{\Sigma}})^*\mathcal G} - (\pi_{|\tilde \Sigma})^* K_{\mathcal G}).
    \]
    Therefore  $S$ is  a non-canonical singularity for $\F$ if, and only if,  $p$ is  a non-canonical singularity for $\mathcal G$.

    From now on, assume that $S$ is a canonical singularity for $\F$.
    Let $v$ be a vector field with isolated singularities defining $\mathcal G= \F_{\Sigma}$ at a neighborhood of $p$.
    If the quotient of eigenvalues of the linear part of $v$ at $p$ is different from $-1$, then $d\omega(p) \neq 0$ and $S$,
    at a sufficiently small neighborhood $U$ of $p$, is a Kupka singularity for $\F$. As was shown in \cite{kupka}, the vector fields annihilating  $d\omega$
    define a smooth codimension two foliation tangent to $\F$, and the projection to $\Sigma$ along the leaves of such foliation
    define a submersion $f :(U,p) \to (\Sigma,p)$
    such that $f^* \mathcal G = \F$.
    If instead the quotient of eigenvalues of $v$ is equal to $-1$, then
    $d \omega$ vanishes identically on $S$, and we cannot apply Kupka's Theorem. Nevertheless
    the same phenomena persists as we now proceed to prove.
    Notice that the coefficients of $\omega$ generate the defining ideal $\mathcal I$ of $S$. After a local change of coordinates, we can assume
    that $S=\{x_1=x_2=0\}$ and that
     \[
         \omega = u ( x_1dx_2 + x_2 dx_1 + \omega_{\ge 2} )
     \]
     where $u$ is a unity and $\omega_{\ge 2}$ have coefficients in $\mathcal I^2$.
     If $\xi$ is a nowhere zero vector field tangent to $S$, say $\xi = \frac{\partial }{\partial x_i}$ for some $i\ge 3$, then the contraction $\omega(\xi)$ belongs to $\mathcal I^2$.
     Therefore we can find a vector field $\xi_{\ge 1}$ with coefficients in $\mathcal I$ such that $\omega(\xi+ \xi_{\ge 1})=0$. Using the local flow of  $\xi + \xi_{\ge 1}$ we   produce a
     submersion  $f_{n-2}: (\mathbb C^n,0) \to (\mathbb C^{n-1},0) $ and a codimension one foliation $\mathcal G_{n-1}$
     such that $\F = f_{n-1}^* \mathcal G_{n-1}$. We proceed inductively to conclude the proof of the proposition.
\end{proof}

\begin{cor}\label{C:deminormal}
Let $\F$ be a codimension one foliation with canonical singularities on a projective manifold $X$.
If $H$ is an invariant algebraic reduced hypersurface, then, outside a closed subset $R\subset H$ of codimension $2$, the  only singularities of $H$
are normal crossing, i.e. $H$ is normal crossing in codimension one. In particular, $H$ is demi-normal in the terminology of \cite[Chapter 5]{kollarsing}.
\end{cor}
\begin{proof}
The fact that $H$ is normal crossing in codimension one follows from the Remark \ref{R:dim2can} and the description of codimension two
canonical singularities given in Proposition \ref{P:codim2}. Since the ambient space is smooth, $H$ (as well as any local complete intersection)
is Cohen-Macaulay and therefore satisfies Serre's  condition $S_2$.
\end{proof}

\subsection{Uniruled foliations with canonical singularites}
From the definition of canonical singularities combined with Theorem \ref{T:MCP},  we obtain the following characterization of uniruledness for
foliations with canonical singularities.

\begin{thm}\label{T:uniruled iff non psef}
Let $\F$ be  a foliation with canonical singularities on a projective manifold. Then $\F$ is uniruled if, and only if,  the canonical bundle of $\F$ is not pseudo-effective.
\end{thm}
\begin{proof}	
If $\KF$ is not pseudo-effective, then Theorem \ref{T:MCP} implies $\F$ is uniruled.

Suppose now that $\F$ is uniruled and with canonical singularities.
We want to prove that $\KF$ is not pseudo-effective.
	
	The uniruledness of $\F$ implies the existence of  a projective manifold $Z$ endowed with a surjective morphism $p:Z \to B$ to another projective manifold $B$  and
	a surjective morphism  $\pi : Z \to  X$ such that the general fibers of $p$ are rational curves, and the image of a general fiber of $p$ under $\pi$ is
	generically tangent to the foliation. After replacing $B$ by a general subvariety, we can further assume that $Z$ and $X$ have  the same dimension.
	
	Let $F$ be a general fiber of $p$. There is no loss of generality in assuming that $\pi(F)$ is a rational curve
	which does not intersect the singular set of $\F$.
	Indeed, if for a general fiber $F$ of $p$ the curve  $\pi(F)$ intersects the singular set of $\F$, then
	there exists a divisor $H$ on $Z$ which dominates $B$ and is mapped to the singular set of $\F$. According to \cite[Theorem VI.1.3]{kollar}, we can find a composition
	of blow-ups  $q:\tilde X\to X$ centered at $\pi(H)\subset \sing(\F)$ such
	that $q^{-1} \circ \pi$ does not contract $H$.
	Since the singularities of $\F$ are canonical, by hypothesis, we have that the canonical bundle of  $q^*\F$
    is the sum of a pseudo-effective line bundle with an effective line bundle, hence still pseudo-effective.
	Replacing $X$ by $\tilde X$ and $Z$ by the elimination of indeterminacies of $q^{-1}\circ \pi$, we arrive
	in a situation where a general fiber $F$ of $p$ is mapped by $\pi$ to the complement of the singular set of $\F$.
		
	Let $L$ be the 	leaf of $\F$ containing $\pi(F)$. Since $\pi(F)$ moves inside $L$ in a family of
    rational curves of dimension at least $\dim(\F)-1=\dim L -1$, it follows that $K_L\cdot \pi(F)<0$. Using that $\pi(F)$ is disjoint from the singular set of $\F$,
	we deduce the identity $K_L \cdot \pi(F) = \KF \cdot \pi(F)$. As $\pi(F)$ moves in a family covering $X$, it follows  that $\KF$ is not pseudo-effective.
\end{proof}

\begin{cor}\label{C:uniruled}
	Let $\F$ be a foliation with numerically trivial canonical bundle on a projective manifold $X$.
	The foliation $\F$ has non-canonical singularities if, and only if, $\F$ is uniruled.
\end{cor}
\begin{proof}
	If $\F$ is uniruled and has canonical singularities, then Theorem \ref{T:uniruled iff non psef} implies $\KF$ is not pseudo-effective.
    Hence, if $\F$ is uniruled and $\KF$ is numerically trivial, then $\F$ cannot have canonical singularities.
		
	Assume now that $\F$ has non-canonical singularities. Then, there exists a projective manifold $Y$, and a  birational morphism
	$\pi:Y\to X$, such that the foliation $\mathcal G = \pi^*\F$ has canonical bundle of the form
	\[
	K_{\mathcal G} = \pi^* \KF + D
	\]
	where $D$ is a non-effective divisor. As $\pi^* \KF$ is numerically trivial, we have that $K_{\mathcal G}$ is pseudo-effective if, and only if,
    $D$ is pseudoeffective. But, according to \cite[Corollary 13]{kollarlarsen}, a contractible divisor is pseudo-effective if, and only if, it is effective.
	Since $D$ is not effective, we deduce that $K_{\mathcal G}$ is not pseudo-effective.
	Theorem \ref{T:MCP}  implies $\mathcal G$ is uniruled, and so is $\F$.
\end{proof}

\begin{cor}
Let $\F$ be a foliation with canonical singularities on a projective manifold $X$.
If $\KF$ is numerically equivalent to zero, then the tangent sheaf of $\F$ is semi-stable
with respect to any polarization of $X$, i.e., for every reflexive subsheaf $\mathcal E \subset T\F$,
we have $c_1(\mathcal E) \cdot H^{n-1} \le 0$ for every ample divisor $H$.
\end{cor}
\begin{proof}
    Fix a polarization $H$ of $X$ and, aiming at a contradiction, assume that $T\F$ is not semi-stable. Consider the Harder-Narasimham filtration of $T\F$ with respect to $H$.
    The maximal semi-stable subsheaf of $T\F$ is closed under Lie brackets (see \cite[Lemma 9.1.3.1]{kol}) and therefore defines
    a foliation $\mathcal G$. Notice that $\mathcal G$ is a subfoliation of $\F$
    with non pseudo-effective canonical bundle since $-c_1(T\mathcal G) \cdot H^{n-1} = K_{\mathcal G} \cdot H^{n-1}<0$.
    It follows that $\mathcal G$ is a foliation with rationally connected general leaf, and $\F$ is uniruled. But at the same time
    $\F$ has canonical singularities and numerically trivial canonical bundle, contradicting Theorem \ref{T:uniruled iff non psef}.
\end{proof}

\begin{cor}\label{C:Frat}
Let $X$ be a uniruled projective manifold, and let $R:X \dashrightarrow R_X$ be the maximal rationally connected meromorphic fibration on $X$.
Let $\F$ be a codimension one foliation with canonical singularities on $X$. If $\KF$ is numerically trivial, then
every leaf of $\F$ dominates $R_X$, i.e. the restriction of $R$ to every  leaf has generically maximal rank. Furthermore, if $\F_{rat}$ is the foliation defined
by $R$, then $\det N\F_{rat}$ is a numerically trivial line bundle.
\end{cor}
\begin{proof}
    Let us call ${\F}_{rat}$ the codimension $q\ (=\dim\ R_X)$ foliation with algebraic leaves induced by  $R:X\dasharrow R_X$.
    After \cite{GHS}, we know that $R_X$ is not uniruled. Therefore, \cite{BDPP} implies  that ${\F}_{rat}$ is given by an holomorphic $q$-form on $X$ without
    zeroes in codimension $1$, and with coefficients in a line bundle $E=\det N \F_{rat}$ such that $E^*$ is pseudo-effective.
    The  restriction of such $q$-form on the leaves of $\F$ defines a non trivial section $\sigma$  of  $\Omega_\F^q\otimes E$, where $\Omega_\F^q$ denotes the
    $q$-th wedge power of the cotangent sheaf of $\F$. The previous Corollary implies that  $T\F$ is semi-stable with respect to
    any polarization of $X$. Therefore, the section $\sigma$ has no zeroes in codimension $1$, and $E = \det N \F_{rat}$ must be numerically trivial.
    On the other hand, any leaf which does not dominate $R_X$ is contained in the zero locus of $\sigma$.
\end{proof}

\subsection{Foliations defined by closed rational 1-forms}
Starting from  Section \ref{S:conormal}, most of this paper will be devoted to prove that a
codimension one foliation with numerically trivial canonical bundle and canonical singularities
is defined by a closed rational $1$-form after an \'etale covering.

In this subsection, we study the polar divisor of foliations with canonical singularities
defined by closed rational $1$-forms. We will make use of the concept  of log canonical pair,
see  \cite{kollarsing} for a thorough treatment of this concept.


\begin{prop}\label{P:logcanonical}
Let $\F$ be a codimension one foliation  on a projective manifold $X$.
Assume that $\F$ is defined by a closed rational $1$-form $\omega$. Let $\Delta=(\omega)_{\infty}$ be the polar divisor of $\omega$ and let $\Delta_{\mathrm{red}}$ be the reduced
divisor with the same support. 
If $\F$ has canonical singularities, then the pair $(X,\Delta_{\mathrm{red}})$
is log canonical.
\end{prop}
\begin{proof}
    Let $\pi:Y \to X$ be a log resolution of the pair $(X,\Delta_{\mathrm{red}})$. As usual let us write
    \[
      K_Y + \pi_*^{-1} \Delta_{\mathrm{red}} = \pi^*(\KX + \Delta_{\mathrm{red}}) + \sum a_i E_i
    \]
    where the sum runs through the exceptional divisors of $\pi$. We want to show that, under our
    assumption, the integers $a_i$ are greater than, or equal to $-1$.

    Let $\G= \pi^* \F$ be the pull-back of $\F$ under $\pi$. The normal bundle of $\F$ is equal
    to $\mathcal O_X( \Delta - Z)$, where $\Delta = (\omega)_{\infty}$ is the polar divisor of $\omega$ and $Z=(\omega)_0$ is the divisorial part  of the  zero set  of $\omega$.  Now $\pi^* \omega$ is a closed rational $1$-form defining $\G$ and therefore
    $N \G = \mathcal O_Y( (\pi^* \omega)_{\infty} - (\pi^*\omega)_0)$. Since $\omega$ is closed, then
    we can write locally
    \[
        \omega = \sum_{i=1}^k \lambda_i \frac{df_i}{f_i} + d \left( \frac{g}{f_1^{n_1} \cdots f_k^{n_k}}\right)
    \]
    where  $g$ is a germ of holomorphic function, $f_i$ are irreducible germs of  holomorphic functions, $\lambda_i$ are complex numbers, and $n_i$ are natural numbers.
    The coefficients $n_i$ correspond to the coefficients of $\Delta- \Delta_{\mathrm{red}}$ on the hypersurface $\{f_i=0\}$. Writing down a similar expression for
    $\pi^*\omega$ we deduce that
    \[
      N \mathcal G = \pi^* ( \Delta - \Delta_{\mathrm{red}} ) + \pi_*^{-1}(\Delta_{\mathrm{red}})  - \pi_*^{-1}(Z)  +  \sum m_i E_i
    \]
    where the sum runs over all exceptional divisors, and $m_i$ are integers not greater than $1$.

    If we compute $K_{\mathcal G} - \pi^* \KF$ using adjunction and the formulas above, we get
    \[
        K_{\mathcal G} - \pi^* \KF =  K_Y + N\mathcal G  - \pi^* (\KX + N\F) =     \sum (a_i + m_i) E_i \, .
    \]
    If $\F$ has canonical singularities, then $0 \le a_i + m_i \le a_i + 1$. Therefore  $a_i \ge -1$, i.e.
    the pair $(X, \Delta_{\mathrm{red}})$ is log canonical.
\end{proof}


\section{The  conormal bundle}\label{S:conormal}

\subsection{From the non-vanishing of cohomology to a foliation by curves}

\begin{lemma}\label{L:h1}
Let $\F$ be a codimension one foliation with numerically trivial canonical bundle on a compact K\"ahler manifold $X$.
Then, for every integer $i$ between $0$ and $n=\dim X$, we have that
\[
    H^i(X,N^* \F)^* \simeq  H^{n-i}(X,\KF) \simeq H^{0}(X,\wedge^i TX\otimes N^*\F).
\]
\end{lemma}
\begin{proof}
	By Serre duality, $H^i(X,N^*\F)^*$ is isomorphic to $H^{n-i}(X,\KX \otimes N\F)$.
    By adjunction, $H^{n-i}(X,\KX\otimes N \F)$ is nothing but $H^{n-i}(X,\KF)$.
    Since $\KF$ admits a flat unitary connection, Hodge theory implies $H^{n-i}(X,\KF)$ is isomorphic to $H^0(X, \Omega^{n-i}_X \otimes \KF^*)$.
	Finally, from the identity  $\Omega^{n-i}_X = \wedge^{i} TX \otimes \KX$, we obtain  that $H^i(X,N^*\F)$ is isomorphic to $H^0(X,\wedge^i TX \otimes N^*\F)$, as claimed.
\end{proof}

Assume that $H^1(X,N^*\F) \neq 0$ and let $v$ be a non-zero twisted vector field with coefficients in $N^*\F$ produced through Lemma \ref{L:h1}.
The twisted vector field $v$ determines a foliation $\mathcal G$ with canonical bundle given by the formula
\begin{equation}\label{E:KG}
    K_{\mathcal G} = N^*\F \otimes \mathcal O_X(-(v)_0) \, .
\end{equation}

If $\omega \in H^0(X,\Omega^1_X \otimes N\F)$ is a twisted $1$-form defining $\F$, then $\omega(v) \in H^0(X,\mathcal O_X)$. Therefore
$\omega(v)$ is either identically zero, or everywhere non-zero. Consequently  the foliation $\mathcal G$ is either contained in $\F$, or it induces a splitting $TX= T\F \oplus T\mathcal G$ of the tangent bundle of $X$.

When the tangent bundle of $X$ splits, then  $\F$ is clearly a smooth foliation.
If instead the contraction of $\omega \in H^0(X,\Omega^1_X \otimes N \F)$ with $v \in H^0(X,TX\otimes N^* \F)$  vanishes identically, i.e.,  the foliation $\mathcal G$ defined by $v$ is tangent to $\F$,
then we have the following result.

\begin{lemma}\label{L:omega(v)=0}
Let $\F$ be a codimension one foliation with numerically trivial canonical bundle on a projective manifold $X$.
Assume that $\F$ has canonical singularities.
If there exists a twisted vector field $v \in H^0(X,TX \otimes N^*\F)$ everywhere tangent to $\F$, then
the canonical bundle of $X$ is pseudo-effective.
\end{lemma}
\begin{proof}
Since $\F$ has canonical singularities, Corollary \ref{C:uniruled} implies that $\F$ is not uniruled.
Therefore, Theorem \ref{T:MBM} implies that the canonical bundle of $\mathcal G$  (the foliation defined by $v$) is pseudo-effective.
But $\KX$ is numerically equivalent to $N^*\F$ by adjunction, and $N^*\F = K_{\mathcal G} \otimes \mathcal O_X((v)_0)$ according to Equation (\ref{E:KG}).
Hence $\KX$ is numerically equivalent to the product of a pseudo-effective line bundle with an effective line bundle. It follows that $\KX$ is
pseudo-effective.
\end{proof}

It will be seen later, in Section \ref{S:uniruled}, that the pseudo-effectiveness of $\KX$ automatically implies that $\F$ is smooth.

\subsection{Sufficient conditions for the non-vanishing of cohomology}

\begin{lemma}\label{L:logdiv}
	Let $\F$ be a codimension one foliation on a projective manifold $X$ 
	Suppose there exists a
	closed analytic subset $R \subset X$ of codimension at least $3$, and  a
	$\mathbb C$-divisor $D=\sum_{i=1}^k  \lambda_i H_i$    supported on $\F$-invariant hypersurfaces  $H_i$
    such that for every $x \in X \setminus  R$, we can locally write
	\begin{equation*}
	\omega\wedge\left(\sum\lambda_i \frac{dh_i}{h_i}\right)=d\omega
	\end{equation*}
	where $h_1, \ldots, h_k$ are local equations  for $H_1, \ldots, H_k$ and $\omega$ is a suitable defining one form of $\F$.
	 Then $c_1(N^* \F )- c_1(D)$ belongs to $H^1(X, {N^*\F})$.
\end{lemma}
\begin{proof}
		One can assume for simplicity that $R=\emptyset$ since  $H^1(X,N^* \F) \simeq H^1(X\setminus R,N^*\F)$.
		Let $\mathcal U = \{ U_{\alpha} \}$ be a sufficiently fine  open covering of $X$. If $\{ g_{\alpha \beta} \in \mathcal O^*(U_{\alpha \beta})\}$
		is a cocycle representing $N \F$, then, according to our hypothesis, the collection of holomorphic $1$-forms $\{ \theta_{\alpha \beta} \in \Omega^1_X(U_{\alpha \beta}) \}$  defined by
		\[
		  \theta_{\alpha\beta} = \frac{dg_{\alpha \beta}}{g_{\alpha \beta}}-\left(\sum\lambda_i\frac{dh_{i,\beta}}{h_{i,\beta}}-\sum\lambda_i\frac{dh_{i,\alpha} }{h_{i,\alpha}}\right)
		\]
		vanishes along the leaves of $\F$. Thus, we have an induced class in  $H^1(X, {N^*\F})$ with image in $H^1(X, \Omega^1_X)$ representing
        $c_1(N^*\F) - \sum \lambda_i c_1(\mathcal O_X(H_i))$.
\end{proof}

When the foliation $\F$ has singular set of codimension at least three, we can apply the lemma above to conclude
that $N^*\F$ is numerically equivalent to zero, or that $H^1(X,N^*\F)\neq 0$. Indeed, under this assumption, we are always in the latter
case as one can see by applying the next lemma in the case $D=0$.

\begin{lemma}\label{L:logdivbis}
	Let $\F$ be a codimension one foliation on a projective manifold $X$ satisfying the assumptions of Lemma \ref{L:logdiv}.
    Moreover, assume that the divisor $D$ has real coefficients strictly greater than $-1$ and that its support is normal crossing in codimension one (see the definition in Corollary \ref{C:deminormal}) . 
	Then $H^1(X, {N^*\F})\not=0$.
\end{lemma}
\begin{proof}
 We will keep the notations used in the  proof of Lemma \ref{L:logdiv}.
    If $D$ is not numerically equivalent to $N^* \F$, then
    the result follows from Lemma \ref{L:logdiv}. Hence, we will assume from now on
    that ${N^*\F}$ is numerically equivalent to $D$.

  By assumption, there exists near every point $x \in X$  a function $\varphi$
    expressed  locally as $\sum_{j=1}^k \varphi_j$ with $\varphi_j= \lambda_j \log {|h_j|}^2$. Notice that the functions
    $\varphi_j$ 
    satisfy the identity  $\frac{i}{2\pi}\partial\overline{\partial}\varphi_j=\lambda_j[H_j]$
    of currents. Moreover, for some appropriate choice of local defining forms $\omega_\alpha$ of $\F$, we get the following equalities
    \begin{equation}\label{E:closedness}
        \omega_{\alpha} \wedge \partial\varphi_{\alpha} =d\omega_{\alpha} \quad \text{ and } \quad
        \partial\varphi_{\alpha}=\sum\lambda_i\frac{dh_i}{h_i} \,
    \end{equation}
    on open subsets $U_{\alpha}$ of the covering $\mathcal U$.

    Since $N^*\F$ is numerically equivalent to $D$, we can interpret $\varphi$ as a local  weight of a
    (singular) metric on $N^* \F$. Therefore the local expressions
    \begin{equation}\label{E:localexpression}
        ie^{\varphi_{\alpha}+\log{|H_\alpha|}}\omega_{\alpha}\wedge\overline{\omega_{\alpha}}
    \end{equation}
    give rise to a  positive $(1,1)$-current $T$ on $V=X\setminus\sing(\F)$ for a suitable choice of $H_\alpha\in{ \mathcal O}^* (U_\alpha)$. Indeed, $e^{\varphi_{\alpha}}$ is locally integrable on $V$, since $\lambda_i>-1$. Note also that $T_\alpha= ie^{\varphi_{\alpha}}\omega_{\alpha}\wedge\overline{\omega_{\alpha}}$ is closed on $U_\alpha \cap V$, thanks to (\ref{E:closedness}).
    Beware that $e^{\varphi_{\alpha}}$ may fail to be integrable near some point of the singular locus. Nevertheless, as $\mbox{sing}\ \F$ has codimension $\geq 2$, $T_\alpha$ extends  uniquely to a closed positive current on $U_\alpha$ (namely the trivial extension of $T_\alpha$), again denoted $T_\alpha$. Observe now that the globally defined positive current $T= |H_\alpha|T_\alpha$ is also closed. Indeed, $i\partial\overline{\partial} T= i\partial\overline{\partial}(|H_\alpha|) T_\alpha$ is identically zero as an exact positive $(2,2)$-current. Because $T_\alpha$ is directed by the foliation, this implies that $|H_\alpha|$ is pluriharmonic,  hence constant on the leaves, whence the closedness of $T$. Note that replacing $\varphi_\alpha$ by $\varphi_\alpha + \log {|H_\alpha| }$ does not affect equality  (\ref{E:closedness}). In particular, one can suppose that $H_\alpha=1$ in (\ref{E:localexpression}).

     Remark that $T$ splits on $V$ as $-\eta\wedge\omega$, where $\eta=ie^{\varphi_{\alpha}}\overline{\omega_{\alpha}}$ is a  $\overline{\partial}$-closed $(0,1)$-current with values in $N^* \F$ globally defined on $V$.

       By assumptions, $D$ is locally normal crossing on $X-S$ where $S$ has codimension $\geq 3$ in $X$. Then the splitting above is indeed defined on $X-S$.


 In particular, if we denote by $B_p$ an open ball centered at $p\in X$, the restriction of $\eta$ to $B_p -S$ is indeed $\overline{\partial}$ exact.
    Therefore, 
    by Mayer-Vietoris,
    the cohomology class $\{\eta\}$ defined in
    $H^1(X-S,N^* \F)$ extends to a  class in $H^1(X,N^* \F)$. The positivity of $T$ implies that this
    class is  non--trivial, and the lemma follows.
\end{proof}

For later use, let us record a consequence of the above. Notice that, in its statement, we refer to the residues
of a $1$-form with coefficients in a flat line bundle. Of course, when the line bundle is not trivial, such residues are not well-defined complex numbers,
but  it  makes sense to ask whether they are zero or not.

\begin{cor}\label{C:closed without residues}
Let $\F$ be a codimension one foliation with canonical singularities.
If $\F$ is defined by a closed rational $1$-form $\eta$ without residues  and with coefficients in a flat line bundle,
then $H^1(X,N^*\F)\neq 0$. In particular, if $\F$ is a fibration, then
$H^1(X,N^*\F) \neq 0$.
\end{cor}
\begin{proof}
    Since $\F$ has canonical singularities, Proposition \ref{P:codim2} implies that the supports of $(\eta)_0$ and $(\eta)_{\infty}$ are disjoint.

    If $x \in |(\eta)_{0}|$,  then there exists a neighborhood $U$ containing $x$
    and a holomorphic primitive of $f: U \to \mathbb C$ of $\eta$, i.e., $\eta_{|U} = df$. If $\omega$ is a holomorphic $1$-form defining $\F_{|U}$
    with zeros of codimension at least two, then $\omega_{|U} = h df$ for some meromorphic function $h : U \to \mathbb P^1$ without zeros. It follows that
    \[
        d \omega = dh \wedge df = \omega \wedge \left( - \frac{dh}{h} \right) \, .
    \]
    Since  the residues of $\frac{dh}{h}$ are all negative, we conclude, using also Corollary
    \ref{C:deminormal},  that $\F$ satisfies the assumptions
    of Lemma  \ref{L:logdivbis} at a neighborhood of  $|(\eta)_{0}|$.

    Let now $x \in |(\eta)_{\infty}|$. Since the supports of $(\eta)_0$ and $(\eta)_{\infty}$ are disjoint
    and  $\eta$ has no residues, we can write  $\eta_{|U} = d(f^{-1})$ where $f$ is a holomorphic function $f:U\to \mathbb C$
    with zero set contained in $|(\eta)_{\infty}|$. As before, we can write $\omega_{|U} = hdf$ for some meromorphic function $h$ without zeros on $U$, and
    conclude that the assumptions of Lemma \ref{L:logdivbis} are satisfied
    everywhere. Therefore  $H^1(X,N^*\F) \neq 0$.

    Finally, if $\F$ is a fibration, then we can take $\eta = df$ where $f:X \to \mathbb P^1$ is any first integral for $\F$.
\end{proof}

\section{Criterion for uniruledness}\label{S:uniruled}

The main goal of this section is to obtain information
about the ambient manifold when  there exists a codimension one
foliation with numerically trivial canonical bundle with non-empty singular set.

\subsection{Pseudo-effectiveness of the canonical bundle implies smoothness}

A particular case of the result below ($ p = \dim X-1$) already appeared in \cite{Touzet3}. The arguments here are a
simple generalization of the arguments therein and heavily rely on an integrability criterion due to Demailly in \cite{Demailly}. They have also high order of contact with the arguments carried out by Bogomolov
in \cite{Bogomolov}, see also \cite[Section 5]{Peter}.

\begin{thm}\label{T:psef} Let  $X$ be a compact K\"{a}hler manifold with $\KX$ pseudo-effective, $L$  a flat line bundle on $X$,  $p$  a positive integer
and $v \in H^0(X,\bigwedge^p TX\otimes L)$  a non-zero section. Then the zero set of  $v$ is empty.
\end{thm}
\begin{proof}
	Let $q = \dim X - p$. From the isomorphism $\bigwedge^p TX  \cong \Omega^q_X \otimes \KX^{*}$ we see that $v$ defines  a twisted $q$-form $\omega \in H^0(X,\Omega^q \otimes \KX^* \otimes L)$.

    The pseudo-effectiveness of $\KX$ implies the existence of
    singular hermitian metric on it with non-negative curvature. The same holds for $\KX\otimes L^*$ by flatness. A metric on $\KX \otimes L^*$ can be identified
    with a section $g$ of $(\KX \otimes L^*) \otimes \overline{(\KX \otimes L^*)}$ and we can use it to define a $(q,q)$-form $\eta$ with $L^{\infty}_{loc}$ coefficients   through the formula $\eta = g(v,v)$.
    Concretely, if $\{ U_\alpha\}$ is an open covering of $X$, then there exists  plurisubharmonic functions $\varphi_\alpha$ on $U_\alpha$ such that
    \[
        |h_{\alpha \beta}|^2 = \exp(\varphi_\alpha - \varphi_\beta) \, .
    \]
    where $h_{\alpha \beta}$ is a cocycle defining  $\KX\otimes L^*$. Thus
    \[
        \eta = i \exp(  \varphi_\alpha) \omega_\alpha \wedge \overline{\omega_\alpha}.
    \]

    Demailly,  in \cite{Demailly}, proved that  the identity of currents $d\omega_\alpha = -\partial \varphi_\alpha \wedge \omega_\alpha$
    holds true, see also the proof of  \cite[Proposition 2.1]{BPT}. {In other words, $\omega$ is $\nabla_g$-closed, where $\nabla_g$ is the Chern connection associated to $g$. Consequently,
    $d \eta=0$  as a current, and $\eta$  defines a class in $H^{q,q}(X,\mathbb C)$.
    Poincar\'e-Serre duality implies the existence of  $[\rho] \in H^{p,p}(X,\mathbb C)$ such that
    \begin{equation}\label{E:poinc}
	    [\eta] \wedge [\rho] \neq 0.
	\end{equation}

    Decompose $\eta$ as the product $\eta=v\wedge g(\overline{v}) $ where $g(\overline{v})=g(\cdot,{v})$ is seen as a   $(0,q)$-form with values in  $\KX\otimes L^*$
    ($ g(\overline{v})=  i \exp( \varphi_\alpha) \overline{\omega_\alpha}$ in a local patch). It is crucial to remark that $g(\overline{v})$ is $\overline{\partial}$-closed, thanks to the $d$-closedness of $\eta$.

   Therefore  the non vanishing of the cup-product (\ref{E:poinc}) can be reformulated as
   \[ [g(\overline{v})]\wedge[v\wedge\rho]\not=0\] where the bilinear pairing involved is $H^{0,q}_{\overline \partial} (\KX\otimes L^*)\otimes H^{0,n-q}_{\overline \partial} (L)\to H^{2n}(X,\mathbb C)$.

    By Hodge symetry, one can choose a  representative of $[ v \wedge \rho ] \in H^{0,n-q}_{\overline \partial} (L)$  of the form $\overline{\gamma}$ 
      where $\gamma$ is a holomorphic $(n-q)$-form valued in the unitary flat bundle $L^*$. In particular $g(\overline{v})\wedge\overline{\gamma}\not=0$, and then, by conjugation, $v\wedge\gamma\not=0$. On the other hand, $v\wedge\gamma$   
      is a section of  $\KX \otimes \KX^*\otimes L\otimes L^*= \mathcal O_X$.
    It follows that  $v$ has no zeros.}
\end{proof}

\begin{thm}\label{T:split}
Let $\mathcal D$ be a distribution of codimension $q$ on a compact K\"{a}hler manifold $X$.
If $c_1(T \mathcal D)=0$ and $\KX$ is pseudo-effective, then $\mathcal D$ is a smooth foliation.
Moreover, there exists a smooth foliation $\G$ on $X$ of dimension $q$
such that  $TX = T\mathcal D \oplus T \mathcal  G$. Finally, if $X$ is projective, then the canonical bundle of $\D$ is torsion.
\end{thm}
\begin{proof}
    The integrability follows from \cite{Demailly}. The previous theorem implies that $\mathrm{sing}(\D) = \emptyset$ and
    that there exists a holomorphic $(n-q)$-form ${\gamma}$ which restricts to a volume form on the leaves of the foliation defined
    by $\mathcal D$.

    In order to prove the result  we just need to modify ${\gamma}$ to obtain that its kernel is the expected complementary subbundle
    defining $\G$.
    This can be done as follows. There is a natural monomorphism of  sheaves
    \[
        \psi:\bigwedge^{n-q-1} T\mathcal D\rightarrow \Omega^1_X,
    \]
    defined by the contraction of  ${\gamma}$ with $n-q-1$ vectors fields tangent to $\D$.
    Notice that the projection morphism of $\Omega^1_X$ onto ${T^*\mathcal D}$ is actually an isomorphism in restriction to $\mathrm{Im}\psi$. Its inverse  provides
    a splitting of the exact sequence
    \[
        0 \to {N^*\mathcal D} \rightarrow\Omega^1_X\rightarrow {T^* \mathcal D} \to 0 \, .
    \]
    Since $\det T^* \D$ is numerically trivial, $\mathrm{Im}\psi$ is an integrable subbundle of $\Omega^1_X$. This subbundle  defines the conormal bundle
    of the sought foliation $\G$.

    Let $L = K \mathcal D^*$ and  $\gamma \in H^0(X, \Omega^{p}_X \otimes L)$, $p=n-q$,  be a twisted $p$-form defining $\G$.
    After passing to a finite \'{e}tale covering, we can assume that the integral Chern class of $L$ is zero, i.e., $L \in \Pic^0(X)$.

    Since $L$ is flat, Hodge symetry implies that $H^0(X, \Omega^{p}_X \otimes L) \cong H^p(X, L^*)$. Let $m = h^p(X,L^*)$ and consider
    the Green-Lazarsfeld set
    \[
        S = \{ E  \in \Pic^0(X) \, | \, h^p(X, E) \ge m \} \, .
    \]
    According to \cite{Simpson}, if $X$ is projective, then  $S$ is a finite union of translates of subtori by torsion points. To conclude the proof of the Theorem, it suffices
    to show that $L^*$ is an isolated point of $S$. Let $\Sigma \subset \Pic^0(X)$ be an  irreducible
    component of $S$ passing through $L$. If $\mathcal P$ is the restriction of the Poincar\'{e} bundle to  $\Sigma \times X$, and $\pi: \Sigma \times X \to \Sigma$
    is the natural projection, then, by semi-continuity, $R^p \pi_* \mathcal P$ is locally free at a neighborhood of $L$. Therefore, we can extend
    the element  $H^p(X,L^*)$ determined by $\gamma$ to a holomorphic family of non-zero elements with coefficients in line bundles $E \in \Sigma$ close
    to $L^*$. Hodge symetry gives us a family of holomorphic $p$-forms with coefficients in the duals of these line bundles. Taking the wedge
    product  of these $p$-forms with a $q$-form defining $\mathcal D$ we obtain, by transversality of $\D$ and $\G$,
    non-zero sections of $H^0(X, \KX \otimes N \mathcal D \otimes E^*)$ for $E$ varying on a small neighborhood of $L^*$ at  $\Sigma$.
    Since $\KX \otimes N \mathcal D \otimes E^* \in \Pic^0(X)$, this implies that $E \in \Sigma$ if,
    and only if, $E=\KX \otimes N \mathcal D = L^*$.
    Thus $\Sigma$ reduces to a point.
\end{proof}

\begin{remark}
    The Theorem above provides evidence toward  the following  conjecture of Sommese (\cite{so}):
    {\it if $\F$ is a smooth foliation of dimension $p$
    with trivial canonical bundle on a compact K\"{a}hler manifold $X$, then there exists a holomorphic $p$-form on $X$ which is non-trivial when
    restricted to the leaves of $\F$.}
\end{remark}


\subsection{Criterion for uniruledness}
Theorem \ref{T:psef} allow us to deduce Theorem \ref{TI:D} of the Introduction which we state again below for convenience.
It  confirms  \cite[Conjecture 4.23]{Peternell}.

\begin{thm}[Theorem \ref{TI:D} of the Introduction]
Let  $X$ be a projective manifold and $L$ be a pseudo-effective line bundle on $X$.
If there exists non trivial $v \in H^0(X,\bigwedge^p TX\otimes L^*)$ vanishing at some point, then
$X$ is uniruled. In particular, if there exists a foliation $\F$ on $X$ with $c_1(T\F)$ pseudo-effective and $\sing(\F) \neq \emptyset$,
then $X$ is uniruled.
\end{thm}
\begin{proof}
    If  $X$ is not uniruled, then $\KX$ is pseudo-effective \cite[Corollary 0.3]{BDPP}.
    Theorem \ref{T:MBM} together with Mehta-Ramanathan Theorem \cite{MR} imply
    that the Harder-Narasimhan filtration of the restriction of   $TX$  to curves obtained as complete intersections of sufficiently ample divisors
    has no subsheaf of positive degree. Therefore the same holds true for $\bigwedge^p TX$, and consequently $L$
    cannot intersect ample divisors positively. This property together with its pseudo-effectiveness implies  $c_1(L) = 0$.
    We can apply Theorem \ref{T:psef}  to conclude that
    $\mathrm{sing}(  v)= \emptyset$, and obtain a contradiction.
\end{proof}

\begin{cor}\label{C:smooth}
Let $\F$ be a codimension one foliation with canonical singularities and numerically trivial canonical bundle.
If $H^1(X,N^*\F)\neq 0$, then $\F$ is smooth.
\end{cor}
\begin{proof}
    If $H^1(X,N^* \F) \neq 0$, then Lemma \ref{L:h1} implies the existence of a nonzero $v \in H^0(X, TX \otimes N^*\F)$.
    Let $\omega \in H^0(X,\Omega^1_X\otimes N\F)$ be a twisted $1$-form defining $\F$. If $\omega(v)\neq0$ in $H^0(X,\mathcal O_X)$, then clearly
    $\F$ is smooth. If instead $\omega(v) = 0$, then Lemma \ref{L:omega(v)=0} implies $\KX$ is pseudo-effective, and Theorem \ref{T:psef} implies that
    $\F$ is also smooth in this case.
\end{proof}

\subsection{Foliations with compact leaves}
In the statement below, by a compact leaf of a foliation we mean a compact invariant submanifold of the same dimension as the foliation
which does not intersect the singular set.

\begin{thm}\label{T:leaves}
Let $\F$ be a codimension $q$ foliation with numerically trivial canonical bundle on a compact K\"{a}hler manifold $X$ of dimension $n>q$.
If $\F$ admits a compact leaf, then $\F$ is smooth, and there exists a smooth foliation $\G$ of dimension $q$ everywhere
transverse to $\F$. Furthermore, if $X$ is a projective manifold, then $\KF$ is a torsion line-bundle.
\end{thm}
\begin{proof}
    Let $\omega$ be a K\"{a}hler form and consider $\xi=\omega^{n-q}$.
    Let $v\in H^0(X,\wedge^{n-q} TX \otimes \KF)$ be a $(n-q)$-vector defining
    $\F$. Contract $\xi$ with $v$ in order to obtain a $\overline \partial$-closed $(0,n-q)$-form $\alpha$ with coefficients
    in $\KF$.
    If this form is not $\overline \partial$-exact, one obtains
    by Hodge symetry a  holomorphic  $(n-q)$-form with coefficients in the dual of  $\KF$
    which is by construction non zero on a general leaf of  $\F$. By adjunction formula, this $(n-q)$-form
    defines a foliation $\G$ totally transverse to $\F$.

    To prove that $\alpha$ is not $\overline{\partial}$-exact, observe that the restriction of $\alpha$ to
    a compact leaf $L$ can be naturally identified with a non-trivial volume form. Indeed, if $(z_1, \ldots, z_{n-q})$
    are local coordinates on $L$, then the restriction of $v$ to $L$ is, up to a scalar multiple, nothing but
    \[
         \left( \frac{\partial }{\partial z_1} \wedge \cdots \wedge \frac{\partial}{\partial z_{n-q}} \right) \otimes ( dz_1 \wedge \cdots \wedge dz_{n-q} )
    \]
     and therefore we can identify $\alpha_{|L}$ with $\xi_{|L}$.

     Finally, when $X$ is projective we can argue exactly as in the proof of Theorem \ref{T:split} in order
     to deduce that $\KF$ is torsion.
\end{proof}

\begin{cor}\label{C:alg leaves}
Let $\F$ be a foliation with numerically trivial canonical bundle on a rationally connected manifold $X$.
If $\F$ has a compact leaf, then $\F$ is the foliation by points.
\end{cor}
\begin{proof}
    First notice that the foliation $\G$ constructed in Theorem \ref{T:leaves} has canonical bundle $\KG$ numerically equivalent
    to $\KX$. Consider now the inclusion $\KG^* \to \wedge^q TX$ and restrict it to a very free rational curve $C$. Because
    $TX_{|C}$ is a direct sum of line bundles of strictly positive degree, this morphism must vanish identically along $C$
    unless $q = \dim X$.  Since the very free rational curves cover $X$, we deduce that $T\G=TX$, i.e. $\G$  is the foliation with just one leaf
    and $\F$ is the foliation by points.
\end{proof}
 When all the leaves are compact, the structure of the foliation is particularly simple:
\begin{thm}\label{TH:smoothalgint}
Let $\F$ be a smooth and algebraically integrable foliation with numerically trivial canonical bundle
on a $n$-dimensional projective manifold $X$. Then, perhaps after passing to
a finite \'{e}tale covering, the manifold $X$ is a product of projective manifolds and $\F$ is defined
by the projection to one of them.
\end{thm}
\subsection{Proof of Theorem \ref{TH:smoothalgint}.}

By \cite[Proposition 2.5]{HV}, $\F$ is induced by a morphism $\varphi:X\to Y$ onto a normal projective variety $Y$ with connected fibers.
Let us state in the first place the following proposition due to Druel.

\begin{prop}\cite[Lemma 6.4,6.5]{Druel2}.\label{P:productbasechange}
Theorem \ref{TH:smoothalgint} holds true whenever $\varphi$ is smooth. Moreover, if the irregularity of every fiber of $\varphi$ is zero, the product structure on an \'etale cover is induced by an appropriate base change: there exists complex manifolds $Y_1$ and $F$, as well as a finite \'etale cover $Y_1\to Y$ such that $Y_1\times_Y X\simeq Y_1\times F$ as varieties over $Y_1$.
\end{prop}

\begin{remark}
In the original statement of \cite{Druel2}, \textit{loc.cit}, it is also assumed that the canonical bundle $K_X$ is pseudo-effective in order to ensure that $\F$ admits a transverse foliation. On the other hand, the existence of this foliation is guaranteed by Theorem \ref{T:leaves}.
Note also, that the existence of a transverse foliation implies that $\varphi$ (assumed to be smooth) is a locally analytically trivial fibration. In particular all  fibers are isomorphic.
\end{remark}


\begin{prop}\label{P: calabiyaucase}
Suppose that a fiber of $\varphi$ is a Calabi-Yau manifold (as the terminology may vary, this means here that its universal cover is compact). Then Theorem \ref{TH:smoothalgint} holds true.
\end{prop}
\begin{proof}
The proof follows \textit{verbatim} the nice arguments developped by Druel in the proof of Proposition 6.6, \textit{loc.cit}. For the sake of convenience of the reader, we recall the strategy and also extract what is really necessary in our setting.

By applying Hwang-Viehweg's \'etale version of Reeb stability theorem \cite[Theorem 2.7]{HV}, there exists a finite set of indices $\Gamma$,  morphisms $g_\gamma:Y_\gamma\to Y$ with $Y_\gamma$  smooth, finite onto their images and such that
\begin{enumerate}
\item $\bigcup_{\gamma\in\Gamma}g_\gamma (Y_\gamma)=Y$.
\item The normalization $X_\gamma$ of $Y_\gamma\times_Y X$ is smooth and  the canonical projection $\varphi_\gamma:X_\gamma\to Y_\gamma$ (which then defines $f_\gamma^*\F$ by the natural smooth \'etale morphism $f_\gamma:X_\gamma\to X$) is smooth projective with connected fibers.
\end{enumerate}

In that way, we obtain, by taking $Y_1$ to be the normalization of $Y$ in the Galois closure of the compositum of the fields $\C(Y_\gamma)$, that the projection $\varphi_1: X_1\to Y_1$ is smooth with connected fibers, where $X_1$ is the (singular) variety defined as the normalization of $Y_1\times_Y X$. Note that $\varphi_1$ defines the foliation $f_1^*\F$ where $f_1:X_1\to X$ is the natural morphism.
Finally, by considering a desingularization $Y_2$ of $Y_1$, we inherit on the smooth projective variety $X_2=Y_2\times_{Y_1}X_1$  a foliation given by the (smooth) projection $\varphi_2:X_2\to Y_2$ which is nothing but the pull-back of $f_1^*\F$ under the natural morphism $f_2:X_2\to X_1$. This construction can be performed from the datum of any regular and algebraically integrable foliation. In the case where $c_1(T_\F)=0$, an elementary but fundamental calculation yields 
 \[
 K_{X_2/Y_2}
 \sim 0 \, ,
 \]
 i.e. $K_{X_2/Y_2}$ is linearly equivalent to zero.

In our case, we have moreover that the fibers $F$ of $\varphi_2$ are Calabi-Yau manifolds. One can thus invoke Proposition \ref{P:productbasechange} to claim that there exists a finite \'etale cover $g_3:Y_3\to Y_2$ such that $X_3:=Y_3\times_{Y_2} X_2\simeq Y_3\times F$. This allows to "go back" in the previous construction  and exhibit a (singular) variety $X_4$ having a product structure  $Y_4\times F$ and equipped with finite dominant morphism $f_4:X_4\to X_1$ such that ${(f_1\circ f_4)}^*\F$ is given by the projection $\varphi_4:X_4\to Y_4$. Indeed, one takes $Y_4$ as the normalization of $Y_1$ in $\C (Y_3)$ and $X_4$ as the normalization of $Y_4\times_{Y_1} X_1.$  By replacing $Y_4$ by a suitable finite cover, one can assume that there exists a finite group $G\in\Aut (Y_4)$ such that $Y\cong Y_4/G$. In particular, $X\cong X_4/G$ for the natural extention of $G$ on the fiber product $X_4$. By construction, $G$ preserves the vertical fibration on $X_4\cong Y_4\times F$. One can then write $g(x,y)=(g_1(x), g_2(x,y))$. The family $g_2(x,.)$ of automorphisms of $F$ depends continuously on $x$, hence is constant thanks to the vanishing of $h^0(F,T_F)$.
In other words,  $G$ acts diagonally on $Y_4\times F$.  Replacing $G$ by some appropriate quotient , there is no loss of generalities in assuming that $G$ injects into $\Aut (Y_4)$ and $\Aut (F)$. This prevents  the existence of a non trivial element of $G$ fixing pointwise an hypersurface in $X_4$. Then the quotient  map $X_4\to X_4/G\cong X$ is \'etale in codimension one, hence \'etale by   Nagata-Zariski purity criterion. This concludes the proof.
\end{proof}

The following statement will give a way to deal with the general case by extracting the abelian factor and reducing to the previous case.
\begin{prop}
Assumptions as in \ref{TH:smoothalgint}. There exists two regular algebraically foliations $\F_1$ and $\F_2$ such that $T_\F=T_{\F_1}\oplus T_{\F_2}$ with the additional properties:
\begin{enumerate}
\item $K_{\F_1}\sim K_{\F_2}\sim 0$
\item The leaves of $\F_2$ are Calabi-Yau manifolds.
\item Up to passing to a finite \'etale cover, $T_{\F_1}$ is holomorphically trivial. In particular $\F_1$ is defined by the orbits of an algebraic subgroup $A$ of ${\Aut}0 (X)$ isomorphic to an abelian variety.
\end{enumerate}
\end{prop}
\begin{proof}
Recall firstly  (cf.\cite{Beauville}) that, for a compact K\"ahler manifold $X$ with $c_1(X)=0$,  the splitting  $TX:=\mathcal E  \oplus{ \mathcal{E}^\perp}$ with respect to a Ricci flat metric, where $\mathcal E$ is the flat factor, is indeed intrinsically defined, i.e. does not depend on the choice of the metric and that, up to finite cover, correspond to a unique decomposition $T\times V$ where $T$ is a compact torus and $V$ a simply connected Calabi-Yau manifold. In particular this splitting is invariant by automorphisms and descend to finite \'etale quotients.

With the notations of the proof of Proposition \ref{P: calabiyaucase}, let us examine on $X_\gamma$ the pull-back foliation ${\F}_\gamma=f_\gamma^*\F$ defined by the smooth projective morphism $\varphi_\gamma$. Each fiber $F$ of $\varphi_\gamma$ is equipped with a canonical splitting of $T_\F$ recalled above.  Note that $X_\gamma$ comes equipped with a transverse foliation ${\mathcal G}_\gamma=f_\gamma^*\mathcal G$ to ${\F}_\gamma$, where $\mathcal G$ is a foliation on $X$ transverse to $\F$ whose existence is guaranteed by Theorem \ref{T:leaves}.  In particular, the fibration $\varphi_\gamma$ is locally trivial from the analytic viewpoint and consequently these fiberwise splittings fit together to produce two regular and algebraically integrable subfoliations ${\mathcal F}_{1,\gamma}$ (corresponding to the flat factor on the fibers), ${\mathcal F}_{ 2,\gamma}$ of ${\mathcal F}_\gamma$ such that
\[T_{{\mathcal F}_{\gamma}}=T_{{\mathcal F}_{ 1,\gamma}}\oplus T_{{\mathcal F}_{2,\gamma}}\]

These two foliations obviously descend to foliations on $f_\alpha (X_\alpha)$ which glue together when $\gamma$ varies and finally define on $X$ two regular algebraically integrable subfoliations $\F_1$, $\F_2$ of $\F$  such that
\[T_{\mathcal F}=T_{{\mathcal F}_1}\oplus T_{{\mathcal F}_ 2}\]

Moreover, observe that on each fiber $F$ of $\varphi_\alpha$, ${T_{\F_1}}|F$ is endowed with a unique flat holomorphic connection ${\nabla_F}$ with finite monodromy coming from the decomposition $T\times V$ on a finite cover. Note that $\nabla_F$ extends uniquely as a flat connection on $T_{\F_1}$ in restriction to an analytic neighborhood of $F$ trivialized as a product $ U\times F$ by the transverse foliation ${\mathcal G}_\gamma$. By uniqueness of $\nabla_F$ on each fiber, these connections glue together to produce a flat connection $\nabla_\gamma$ which descends to $f_\gamma (X_\gamma)$. For the same reasons of uniqueness and the fact that the ${\mathcal G}_\gamma$ come from the same foliation $\mathcal G$ on $X$, the flat connections ${f_\gamma}_*\nabla_\gamma,\ \gamma\in\Gamma$ glue together, hence give rise to a flat holomorphic connection $\nabla$ on $T_{{\F}_1}$. In particular $K_{\F_1}\sim 0$ and automatically $K_{\F_2}\sim 0$ thanks to the above splitting.

We claim now that $\nabla$ has finite monodromy. Remark that it suffices to show that for some $\gamma\in\Gamma$, ${f_\gamma}_* {\nabla}_\gamma$ (or equivalently $\nabla_\gamma$) has finite monodromy, since  $\pi_1 (f_\gamma (X_\gamma))$ surjects onto $\pi_1(X)$.
For this purpose, observe that, by fixing some $x\in Y_\gamma$ and considering $F={\varphi_\gamma}^{-1} (x)$,  $\G_\gamma$ defines a representation (the holonomy representation)
\[\rho:\pi_1(x,Y_\gamma)\to \Aut (F)\]
by lifting loops on leaves of ${\mathcal G}_\gamma$.

Because the morphism $\varphi_\gamma$ is projective, the image $G$ of $\rho$ has to fix some K\" ahler class on $F$, hence contains a subgroup of $\Aut^0 (F)$ of finite index. Hence, by replacing $X_\gamma$ by ${\tilde Y}_\gamma\times_{Y_\gamma} X_\gamma$,  with ${\tilde Y}_\gamma$ some suitable finite \'etale cover of $Y_ \gamma$, one can assume without any loss of generalities that $G\subset \Aut^0 (F)$, this latter being an Abelian variety. Take $g\in G$, $g=\exp(tX)$ for $X\in H^0(F,T_F)$, $t\in\C$. Denote by $L_y$ the stalk at $y\in F$ of the local system $L$ associated to $\nabla_F$ and consider $u\in L_y$. Because the commutation relation $[X,L]=0$ holds ($L$ lifts as a constant local system on $T\times V$), one can infer that $g_* u\in L_{g (y)}$ must coincide with the analytic continuation of $u$ along $s\to\text{exp}(\alpha (s)X)(y)$ where $\alpha:[0,1]\to \C$ is a continuous path joining $0$ to $t$.

Thanks to the exact sequence of the fibration $\varphi_\gamma$, we can  easily conclude that the monodromy group of $\nabla_\gamma$ coincides with that of its restriction $\nabla_F$ and in particular, is finite. This finishes the proof of the Proposition.
\end{proof}

By \cite[Proof of Theorem 1.2]{Brion} (see also \cite[proof of proposition 6.6]{Druel2}), up to replacing $X$ by a finite \'etale cover, one can suppose that $X=A\times Y$, $A$ Abelian variety where $\F_1$ is defined by the projection ${pr}_Y:X\to Y$. Note that $\F_2$, whose leaves have irregularity $0$, is automatically tangent to the second projection ${pr}_A:X\to A$. Moreover,  the action of $A$ on each leaf $\mathcal L$ of $\F$ has to preserve ${T_{\F_2}}|\mathcal L$.
This implies that ${\F_2}$ projects via ${pr}_Y$ as a foliation on $Y$. This last property, combined with Proposition \ref{P: calabiyaucase} allows to conclude the proof of Theorem \ref{TH:smoothalgint}.\qed

\subsection{Algebraic leaves of singular foliations} A generalization of the argument used to prove Theorem \ref{T:leaves}
imposes constraints on algebraic leaves of singular foliations with trivial canonical class.

\begin{thm}\label{T:alg leaves}
Let $\F$ be a foliation with numerically trivial canonical bundle on a projective manifold $X$.
Assume $\sing(\F) \neq \emptyset$. If $L$ is an algebraic leaf of $\F$, then  the Zariski closure of $L$ is uniruled.
\end{thm}
\begin{proof}
    Since $\F$ is singular and $X$ is smooth, $\F$ has dimension and codimension different from zero.
    Theorem \ref{T:leaves} implies that $L$ is not compact. Let $\overline L$ be the
    closure of $L$ and $n : Y \to \overline L$ be its normalization.

    According to \cite[Proposition 4.5]{ADK} (see also \cite[Definition 3.4 and Lemma 3.5]{AD})   there exists an effective Weil divisor $\Delta$ on $Y$ such that
    $n^* \KF$ is linearly equivalent to $\KY + \Delta$. In the terminology of \cite{AD}, the pair $(Y,\Delta)$ is
    a log leaf of $\F$. Let $m : Z \to Y$ be a resolution of singularities of $Y$.

    If $\Delta >0$, then $\KZ$ is not pseudo-effective. Indeed, if $H$ is a very ample divisor on $Y$ and $d$ is the dimension of
    $Y$, then $H^{d-1} \cdot \Delta >0$ and $H^{d-1}$ can be represented by a curve in $Y$ which does not intersect the singular locus
    of $Y$, or even better the centers of $m$. Thus $\KZ \cdot (m^*H)^{d-1} = \KY \cdot H^{d-1} < 0$. Since $m^* H$ is big and nef, it follows
    that $\KZ$ is not pseudo-effective.   Thus by \cite{BDPP} $Z$ is uniruled and the same holds for $\overline L$.
    Similarly, if $\Delta =0$ and the singularities of $Y$ are not canonical, then $\KZ$ is also not pseudo-effective, and we conclude as
    before.

    If $\Delta =0$ and the singularities of $Y$ are canonical, then $m^*\KY$ injects into $\KZ$. Let $q$ be the codimension of $\F$ and
    $\alpha$ be the $(0,n-q)$-form  with coefficients in $\KF$ constructed in the proof of Theorem \ref{T:leaves}, i.e. $\alpha$ is the
    contraction of the $(n-q)$-th power of a K\"{a}hler form $\omega$ with a $(n-q)$-vector field $v \in H^0(X, \wedge^{n-q} TX \otimes \KF)$ defining $\F$. In order to conclude
    the proof, we will show that $\alpha$ is non-trivial in cohomology, what implies that $\F$ is smooth (by the  proof of Theorem \ref{T:leaves}) contradicting our hypothesis.
    For that, let $p = n \circ m$
    be the composition of the normalization with the resolution of singularities.  The pull-back $p^* \alpha$ is a $(0,n-q)$-form with coefficients in $p^* \KF = m^* \KY$.
    Since $m^*\KY$ injects into $\KZ$, it yields a $(0,n-q)$-form on $Z$ with coefficients in $\KZ$. Thus we can regard $p^* \alpha$ as a $(n-q,n-q)$-form on $Z$. Away from the union of the critical locus of $p$ with the pre-image
    of the singular locus of $\F$,   we have, as in the proof of Theorem \ref{T:leaves}, the identity   $p^*\alpha=p^* \omega^{n-q}$ up to replacing $v$ by a suitable scalar multiple. By continuity, this equality holds on the whole $Z$. Hence,  $p^*\alpha$ is a  non-trivial semi-positive form what guarantees the non-triviality of  $\alpha$  in cohomology.
\end{proof}

\begin{thm}\label{T:fibration}
Let $\F$ be a foliation with numerically trivial canonical bundle and canonical singularities
on a projective manifold $X$. If the general leaf of $\F$ is algebraic, then, perhaps after passing to
a finite \'{e}tale covering, the manifold $X$ is a product of projective manifolds and $\F$ is defined
by the projection to one of them.
\end{thm}
\begin{proof}
    Since $\KF$ is numerically trivial and the general leaf is algebraic, Theorem \ref{T:alg leaves} combined
    Corollary \ref{C:uniruled} implies that $\F$ is a smooth foliation. We can thus apply Theorem \ref{TH:smoothalgint}
  in order to conclude.
\end{proof}



\section{Deformations of free morphisms}\label{S:free morphisms}
The goal of this section is to prove the following result.

\begin{thm}\label{T:transproj}
Let $\F$ be a codimension one foliation having at worst canonical singularities on an uniruled projective manifold $X$.
If the canonical bundle of $\F$ is numerically trivial, then $\F$ is a transversely projective foliation.
\end{thm}

Before dealing with the proof of this result, let us recall the definition of transversely projective foliation following \cite{RH}, see also \cite{Croco1,Croco2}.
A {\it transversely projective} structure for a codimension one foliation  $\F$ on a projective manifold $X$ is the data $(P, \mathcal H, \sigma)$  of a $\P^1$-bundle $P\to X$,
a Riccati foliation $\mathcal H$ on $P$,
and a meromorphic section $\sigma:X\dashrightarrow P$
such that $\sigma^* \mathcal H = \F$. Notice that this last condition implies that $\sigma$  is generically transverse to  $\mathcal H$.

Another triple $(P',\mathcal H',\sigma')$  defines the same transversely projective
structure if it is derived from the initial one by a birational bundle
transformation $P\dashrightarrow P'$.
Up to such birational bundle transformations, one can always
assume that $P$ is the trivial bundle $X\times\P^1$ with vertical coordinate $z$, and $\sigma$ is the   section $\{z=0\}$.
The foliation $\mathcal H$ is therefore defined by a Riccati $1$-form
\begin{equation*}
\omega = dz+\omega_0+\omega_1 z+\omega_2 z^2
\end{equation*}
where $\omega_0,\omega_1,\omega_2$ are rational $1$-forms
on $X$. The integrability of $\mathcal H$, $\omega \wedge d \omega=0$,  is equivalent to the equations
\begin{equation}\label{eq:FlatnessCondition}
    \left\{\begin{matrix}
        d\omega_0=\hfill\omega_0\wedge\omega_1\\
        d\omega_1=2\omega_0\wedge\omega_2\\
        d\omega_2=\hfill\omega_1\wedge\omega_2
    \end{matrix}\right.
\end{equation}
Since we have normalized the section $\sigma$ to $\{ z=0\}$,  then   $\F$ is defined by the rational $1$-form $\omega_0$.
Hence, a  foliation $\F$ on $\P^n$ is transversely
projective if, and only if, there exist rational $1$-forms
$\omega_0,\omega_1,\omega_2$ on $X$
satisfying (\ref{eq:FlatnessCondition}) where $\omega_0$
defines the foliation $\F$; we recognize the definition of transversely projective foliations given in \cite{Scardua}.

If there exists a transversely projective structure $(P,\mathcal H, \sigma)$ for $\F$ in which $\omega_2=0$ (or
equivalently  there exists a section $\tilde \sigma : X \dashrightarrow P$ invariant by $\mathcal H$),  then we say that
$\F$ is a {transversely affine foliation}.

If  there exists a transversely projective structure $(P,\mathcal H, \sigma)$ for $\F$ in which $\omega_2=\omega_1=0$, then
$\F$ is a {transversely Euclidean foliation}. In other words, a foliation $\F$ is transversely Euclidean if, and
only if, $\F$ can be defined by a closed rational $1$-form.



Transversely projective foliations behave rather nicely with respect to dominant rational maps
as the lemma below shows.

\begin{lemma}\label{L:empurra}
	Let $F: Y \dashrightarrow X$ be a dominant rational map between  projective manifolds,  $\F$
	be a codimension one foliation  on $X$, and $\mathcal G= F^* \F$ be  the foliation  induced
	by $\F$ on $Y$. Then the following assertions hold true.
	\begin{enumerate}
		\item The foliation $\mathcal G$ is transversely projective if, and only if, $\F$ is transversely projective.
		\item The foliation $\mathcal G$ is transversely affine if, and only if, $\F$ is transversely affine.
	\end{enumerate}
\end{lemma}
\begin{proof}
	If $\F$ is transversely projective (resp. affine or Euclidean), then $\mathcal G=F^* \F$ is
	transversely projective (resp. affine or Euclidean) since such a structure $(P,\mathcal H, \sigma)$ for
	$\F$ pulls back to a similar structure $(F^* P , F^* \mathcal H, F^* \sigma)$ for $\mathcal G$.
	
	Suppose now that $\mathcal G$ is transversely projective (resp. affine). Restrict $\mathcal G$ and its projective structure
	to a sufficiently general submanifold having the same dimension as $X$.  This reduces the problem to case where
	$F$ is a generically finite rational map, and we can apply \cite[Lemme 2.1, Lemme 3.1]{Casale} to conclude.
\end{proof}

We note that, if $\mathcal G$ is transversely Euclidean, then  $\F$ is not necessarily transversely Euclidean.
The simplest examples are linear foliations on Abelian surfaces
which are invariant by multiplication by $-1$ while the defining $1$-forms are not. By taking the  quotient,
we obtain foliations which are transversely affine, but not transversely Euclidean.

\subsection{Deformation of free morphisms}
Let $X$ be a projective manifold of dimension $n$. The morphisms from $\mathbb P^1$ to $X$ are parametrized by a
locally Noetherian scheme $\Mor(\mathbb P^1, X)$
\cite[Theorem I.1.10]{kollar}.
The Zariski tangent space of $\Mor(\mathbb P^1, X)$ at a given morphism $f: \mathbb P^1 \to X$ is canonically
identified with $H^0(\mathbb P^1,f^* TX)$ \cite[Proposition 2.4]{Debarre} \cite[Theorem I.2.16]{kollar}. To understand this,
suppose $\Mor(\mathbb P^1, X)$ is smooth at a point $[f]$, and let $\gamma :(\mathbb C,0) \to \Mor(\mathbb P^1,X)$ be a
germ of holomorphic curve   in  $\Mor(\mathbb P^1, X)$ such that $\gamma(0)= [f]$.
If we fix $x \in \mathbb P^1$ and compute $\gamma'(0)(x)$, we obtain a vector at $T_{f(x)} X \simeq (f^* TX)_x$.  Thus $\gamma'(0) \in H^0(\mathbb P^1, f^* TX)$.

For an arbitrary morphism $f$, the local structure of $\Mor(\mathbb P^1, X)$
at a neighborhood of $[f]$ can be rather nasty, but if $h^1(X, f^* TX) = 0$, then $\Mor(\mathbb P^1, X)$ is smooth
and has dimension $h^0(\mathbb P^1, f^* TX)$ at a neighborhood of $[f]$, see \cite[Theorem I.2.16]{kollar} or  \cite[Theorem 2.6]{Debarre}.

If $[f] \in \Mor( \mathbb P^1, X)$, then  Birkhoff-Grothendieck's Theorem implies that $f^*TX$ splits as a sum of line bundles
$\mathcal O_{\mathbb P^1}(a_1) \oplus \mathcal O_{\mathbb P^1}(a_2) \oplus \cdots \oplus \mathcal O_{\mathbb P^1}(a_n)$
with $a_1 \ge a_2 \ge \cdots \ge a_n$. The morphism $f$  is called free when $a_n \ge 0$. Notice that $h^1(\mathbb P^1, f^* TX) = 0$ when $f$
is a free morphism. Therefore  $\Mor(\mathbb P^1, X)$ is smooth of dimension $h^0(\mathbb P^1, f^*TX) = n + \sum_{i=1}^n a_i$ at a neighborhood of $[f]$.

The scheme $\Mor(\mathbb P^1, X)$ comes together with an evaluation map
\begin{align*}
    F: \mathbb P^1 \times \Mor(\mathbb P^1, X) &\longrightarrow X \\
    (x, [f]) &\longmapsto f(x) \, .
\end{align*}

Let $f$ be a free morphism and $M =M_f$ be the irreducible component of $\Mor(\mathbb P^1, X)$ containing $[f]$.
The evaluation map $F$  has maximal rank at any point of
a neighborhood of $\mathbb P^1 \times \{[f]\}$ in $\mathbb P^1 \times M$ \cite[Corollary II.3.5.4]{kollar}. Indeed, it has maximal rank at a neighborhood
of any point of the $\Aut(\mathbb P^1)$-orbit of $[f]$ under the  action of $\Aut(\mathbb P^1)$ on $\Mor(\mathbb P^1,X)$ defined
by right composition.

\subsection{Tangential foliation on the space of morphisms} Let $\F$ be a foliation on $X$ (not necessarily of codimension one).  We will say that a
germ of  deformation $f_t : \mathbb P^1 \to X$, $t \in (\mathbb C,0)$,  of a  free morphism $f=f_0: \mathbb P^1 \to X$
is tangent to $\F$ if the curves $f_t(x): (\mathbb C,0) \to X$ are tangent to the foliation $\F$ for every $x$ in $\mathbb P^1$.
These deformations  correspond to germs of curves on $M $ tangent to a foliation $\F_{tang}$ on $M $
which we will call the tangential foliation  of  $\F$.

The construction of $\F_{tang}$ is rather simple. Since
$T  M \simeq \pi_* F^* TX$, where $\pi: \mathbb P^1 \times M \to M$ is the natural projection,  the inclusion $T \F \hookrightarrow TX$ gives rise to a morphism
$\pi_* F^* T \F \to T M$. If $\mathcal I$ denotes its image, then we define  $\F_{tang}$ as the foliation
on $M$ determined by the saturation of $\mathcal I$ inside $T M $, i.e., $T \F_{tang}$ is the smallest
subsheaf of $T M $ containing $\mathcal I$ and with torsion free cokernel. The involutiveness of $T \F_{tang}$ follows easily from the
involutiveness of $T\F$ as verified below.

\begin{prop}
The sheaf $T\F_{tang} \subset TM$ is closed under Lie brackets.
\end{prop}
\begin{proof}
It suffices to verify at a neighborhood of a general morphism $[g] \in M $. We can assume for instance
that the image of $g$ is disjoint from the singular set of $\mathcal F$, thus $\mathcal F$ foliates
a neighborhood of $g(\mathbb P^1)$. If $\xi_1, \xi_2$ are germs of  sections of $T \mathcal F_{tang}$ at $[g]$,  then  the orbits of the corresponding
vector fields give rise to deformations of morphisms
$\phi_{1}, \phi_{2}: (\mathbb C,0) \times \mathbb P^1 \to X$ with  $\phi_{i}(0,x) g(x)$ and $\partial_t \phi_{i}(t,x) \in T_{\phi_i(t,x)} \F$ for $i=1,2$. The involutiveness of $\mathcal F$  implies
that $[ \partial_t \phi_{1}(0,x) , \partial_t \phi_{2,t}(0,x) ] \in  T_x \F$ for $x \in \P^1$. The involutiveness of $T\F_{tang}$ follows.
\end{proof}

\begin{remark}
    The idea of studying a foliation through the induced foliations on the space of morphisms is not new, and can be traced backed to
    \cite{Miyaoka}. While it is explored, there, to prove the uniruledness of the ambient manifold, here we will explore the
    uniruledness of the ambient manifold (through the existence of free rational curves) to unravel the structure of the original foliation.
\end{remark}

\subsection{Tangential foliation and transverse structures}

Consider a complex manifold $Y$ endowed with a proper morphism $\pi:Y \to Z$ with connected fibers  to another complex manifold $Z$.
Given a foliation $\mathcal G$ on $Y$, we define the direct image of $\mathcal G$ under $\pi$ as the foliation $\pi_* \G$
on $Z$ with tangent sheaf  given by the saturation in $TZ$ of the image of the natural morphism
\[
\pi_* T\G \to \pi_* TY \to TZ
\]
induced by the composition of the inclusion of $T\G$ in $TY$ with the differential of $\pi$.

Let now $X$ be a uniruled projective manifold carrying a foliation $\F$, and let  $M$ be an irreducible open subset of $\Mor(\mathbb P^1, X)$ formed by free rational curves.
If we consider on $\mathbb P^1 \times M$  the foliation $\mathcal G$ defined as the pull-back of $\F$  under the evaluation morphism, and the codimension one foliation
$\mathcal H$ defined by the fibers of the projection $\mathbb P^1 \times M \to \mathbb P^1$, then the tangential foliation of $\F$ is nothing but the direct image
of the intersection of $\mathcal G$ with $\mathcal H$, i.e.
\[
\F_{tang} = \pi_*( \mathcal G \cap \mathcal H) \, .
\]

When the direct image of a foliation  is non-trivial (i.e. its dimension is different from zero),  then our next result shows that the transverse structure of
the original foliation is constrained.

\begin{thm}\label{T:6.4}
Let $M$ be an algebraic manifold, let $\mathcal G$ be a codimension one foliation on $\mathbb P^1 \times M$,
let $\pi:\mathbb P^1 \times M \to M$ be the natural projection, and let $\mathcal H$ be the codimension one foliation defined by
the fibers of the other natural projection $\rho: \mathbb P^1 \times M \to \mathbb P^1$.
If the general fiber of $\pi$ is generically transverse to $\mathcal G$, and the general leaf of the direct image $\mathcal T= \pi_*(\mathcal G \cap \mathcal H)$
is Zariski dense,  then the codimension of $\mathcal T$ is at most three. Moreover,
\begin{enumerate}
    \item If $\codim \mathcal T =1$, then $\mathcal G$ is defined by a closed rational $1$-form (i.e is transversely Euclidean);
    \item If $\codim \mathcal T=2$, then $\mathcal G$ is transversely affine;
    \item If $\codim \mathcal T=3$, then $\mathcal G$ is transversely projective.
\end{enumerate}
\end{thm}
\begin{proof}
    After replacing $M$ by a Zariski open subset, we can assume that $N\mathcal G$ is of the form $\rho^* \mathcal O_{\mathbb P^1}(\nu+2)$ for
    some non-negative integer $\nu$, and $\mathcal G$ is defined by a rational $1$-form $\Theta$ which can be written as
    \[
        \Theta = (\sum_{i=0}^{\nu} b_i(x) z^i) dz + \sum_{i=0}^{\nu+2} \theta_i z^i
    \]
    where $b_i$ are regular functions on $M$ and $\theta_i$ are holomorphic $1$-forms on $M$. Notice that $\mathcal T$ is the foliation defined by the
    $1$-forms $\theta_i$. If $\Omega$ is the quotient of $\Theta$ by $(\sum_{i=0}^{\nu} b_i(x) z^i)$, then
    \[
        \Omega = dz + \sum_{i\ge i_0} \omega_i z^i
    \]
    where $i_0$ is an integer and $\omega_i$ are rational $1$-forms on $M$. Of course, $\Omega$ is a rational $1$-form defining $\mathcal G$ and the foliation $\mathcal T$ is
    defined by the $1$-forms $\omega_i$.

    The integrability  condition  takes the particularly simple form
    \[
        \Omega \wedge d \Omega = \sum z^{i} dz \wedge d \omega_i + \sum j z^ {i+j-1} \omega_i \wedge dz \wedge \omega_j + \sum z^{i+j} \omega_i \wedge d\omega_j =0 \, .
    \]
    In particular, looking at the coefficient of $z^k dz$, we obtain the existence of constants $\lambda_{ij}^{(k)}$ such that
    \begin{equation}\label{E:chave}
        d \omega_k = \sum \lambda_{ij}^{(k)} \omega_i \wedge \omega_j \,
    \end{equation}
    for any $k \ge i_0$.

    Let $q$ be the codimension of $\mathcal T$, and $p= \dim M-q$ be the dimension of $\mathcal T$.
    Choose $q$ $1$-forms $\alpha_1, \ldots, \alpha_q$ among the $1$-forms $\omega_i$ such that $\alpha_1, \ldots, \alpha_q$ define $\mathcal T$.
    We claim that any of the $1$-forms $\omega_i$ can be
    written as linear combination with constant coefficients of the $1$-forms $\alpha_1, \ldots, \alpha_q$.
    Indeed, for any fixed $k$ there exits rational functions $b_1, \ldots, b_q$ such that
    \[
        \omega_k = \sum_{i=1}^q  b_i \alpha_i \, .
    \]
    Differentiating this expression, taking the wedge product of it with the $(q-1)$-forms $\beta_j = \alpha_1 \wedge \ldots \wedge \widehat{\alpha_j} \wedge \ldots \wedge
    \alpha_q$, and using equation (\ref{E:chave}),  we obtain that
    \[
        db_j \wedge \alpha_j \wedge \beta_j = 0 \implies db_j \wedge \alpha_1 \wedge \cdots \wedge \alpha_q =0 \, .
    \]
    Hence $b_j$ is constant along the leaves of $\mathcal T$, and since the general leaf is Zariski dense, $b_j$ must be constant.
    We conclude the existence of constants $\mu_{ij}^{(k)}$ such that
    \begin{equation}\label{E:chave2}
        d \alpha_k = \sum \mu_{ij}^{(k)} \alpha_i \wedge \alpha_j \, .
    \end{equation}
 Choose $p=\dim M -q$ rational functions $h_1, \ldots, h_p$ on $M$ such that the product
    \[
        \alpha_1 \wedge \ldots \wedge \alpha_q \wedge dh_1 \wedge \ldots \wedge dh_p \neq0
    \]
    does not vanish identically. Let $v_1, \ldots, v_q$ be the unique rational vector fields on $M$ satisfying
    \[
        \alpha_i ( v_j ) = \delta_{ij} \quad \text{ and } \quad dh_i (v_j ) =0 \, ,
    \]
    where $\delta_{ij} $ is the Kroenecker delta.
    Notice that the vector fields $v_1, \ldots, v_q$  satisfy a dual version of equation (\ref{E:chave2}), that is
    \begin{equation*}
        [v_i,v_j] = \sum \mu_{ij}^{(k)} v_k \, .
    \end{equation*}
    Let $\mathfrak v$ be the $q$-dimensional Lie algebra defined by $v_1, \ldots, v_q$.

    We claim that there exists an injective morphism of Lie algebras  from $\mathfrak v$ to the Lie
    algebra of rational vector fields on $\mathbb P^1$. Indeed, for each $i$ there exists a unique
    lift $\hat{v}_i$ of $v_i$ to $\mathbb P^1 \times M$ tangent to $\mathcal G$. More precisely,
    $\hat{v}_i$ is the unique lift of $v_i$ such that $\Omega(\hat{v}_i)=0$. Due to the particular form
    of $\Omega$, it follows that we can write
    \[
        \hat{v}_i = f_i(z) \frac{\partial}{\partial z} + v_i \, ,
    \]
    where $f_i \in \mathbb C(z)$ is a rational function.
    Notice that
    \[
        [\hat{v}_i,\hat{v}_j] =  [f_i(z)\frac{\partial}{\partial z},f_j(z) \frac{\partial}{\partial z}] + [v_i,v_j] .
    \]
    Since $\Omega$ is integrable it follows that the (unique) lift of $[v_i,v_j]$ tangent to $\mathcal G$ is given by
    \[
        \sum \mu_{ij}^{(k)} f_k(z)\frac{\partial}{\partial z} + [v_i,v_j]\,
    \]
    and it must coincide with $[\hat{v}_i,\hat{v}_j]$.
    Hence the map that sends $v_i$ to $f_i(z) \frac{\partial}{\partial z}$ is the sought injective morphism of Lie algebras
    from $\mathfrak v$ to $\mathbb C(z) \frac{\partial}{\partial z}$.

    A classical result of Lie (cf. \cite[Theorem 1.1.1]{Frankpseudo}) says that a finite dimensional Lie subalgebra of
    $\mathbb C(z) \frac{\partial}{\partial z}$ has dimension at most three. Moreover, if its dimension is two, then it is
    isomorphic to the  affine Lie algebra $\mathfrak{aff}(\mathbb C)$, and if its dimension is three, then it its isomorphic to the projective Lie algebra $\mathfrak{sl}(2,\mathbb C)$.
    Therefore $\mathfrak v$ has dimension at
    most three, and consequently the codimension of $\mathcal T$ is at most three.

    If the codimension of $\mathcal T$ is equal to one, then we can write $\Omega =  dz + b(z) \alpha_1$ for suitable $b \in \mathbb C(z)$ and we see that $\Omega/b(z)$ is a closed
    $1$-form defining $\mathcal G$.

    If the codimension of $\mathcal T$ is equal to two, then  there exists
    $v_1, v_2 \in \mathfrak v$ satisfying $[v_1,v_2] = v_1$.
    Let $a(z)  \frac{\partial}{\partial z}$ be the image of $v_1$ in $\mathbb C(z) \frac{\partial}{\partial z}$ and
    $b(z) \frac{\partial}{\partial z}$ be the image of $v_2$. Therefore
    \[
        [a(z)  \frac{\partial}{\partial z},b(z)  \frac{\partial}{\partial z}]= a(z)  \frac{\partial}{\partial z} \implies ab' -ba' = a  \, .
    \]
    If $\varphi : \mathbb P^1 \to \mathbb P^1$ is the rational map $\varphi(z) = -b(z)/a(z)$, then the above equation implies
      \[
         \varphi^*  \frac{\partial}{\partial z} = -\frac{1}{(b/a)' (z)}  \frac{\partial}{\partial z} = -\frac{a^2}{(a'b-ab')}  \frac{\partial}{\partial z} = a  \frac{\partial}{\partial z} \, .
    \]
    Similarly $\varphi^*  z \frac{\partial}{\partial z}= b(z)  \frac{\partial}{\partial z}$.  If we consider on $\mathbb P^1\times M$
    the foliation $\mathcal R$ generated by $\hat{\mathcal T} = (\pi^* \mathcal T )\cap \mathcal H$, $w_1 = \frac{\partial}{\partial z} + v_1$
    and $w_2 = z\frac{\partial}{\partial z} + v_2$, then $\mathcal G = (\varphi \times \mathrm{id}_M)^* \mathcal R$. Notice that the foliation
    $\mathcal R$ is a Riccati foliation on $\mathbb P^1 \times M$ with the section at infinity invariant, and thus it is tranversely affine.

    If the codimension of $\mathcal T$ is equal to three, then  there exists
    a basis $v_1,v_2,v_3$ of $\mathfrak v$ satisfying $[v_1,v_2]=v_1$, $[v_1,v_3]=2v_2$, and  $[v_2,v_3]=v_3$. We write, analogously to the previous case,
    the images of $v_1, v_2,$ and $v_3$ in $\mathbb C(z) \frac{\partial}{\partial z}$ as $a(z)  \frac{\partial}{\partial z},  b(z) \frac{\partial}{\partial z},$
    and $c(z) \frac{\partial}{\partial z}$ respectively. We still get the identity $ab'-ba'=a$ and also get the identities $bc'-b'c = 2b$ and
    $bc' - b'c =c$. If we  consider the rational map $\varphi(z) = -b(z)/a(z)$,
    then we have that $ \varphi^*  \frac{\partial}{\partial z}  = a(z) \frac{\partial}{\partial z}$,  $ \varphi^*  z\frac{\partial}{\partial z}  = b(z) \frac{\partial}{\partial z}$,
    and  $ \varphi^*  z^2\frac{\partial}{\partial z}  = c(z) \frac{\partial}{\partial z}$. As before, we  obtain  the existence of a Riccati foliation $\mathcal R$ on $\mathbb P^1 \times M$
    with tangent sheaf generated by $\hat{\mathcal T} = (\pi^* \mathcal T )\cap \mathcal H$, $v_1 +\frac{\partial}{\partial z}$, $v_2 + z\frac{\partial}{\partial z}$, and $v_3 + z^2 \frac{\partial}{\partial z}$ such that
    $\mathcal G =   (\varphi \times \mathrm{id}_M)^* \mathcal R$. It follows that $\mathcal G$ is transversely projective.
\end{proof}

\subsection{Another uniruledness criterion}
The next two results hold for foliations of any codimension.

\begin{prop}\label{P:subfol}
	Let $[g] \in M  \subset \Mor(\mathbb P^1,X)$ be a general free morphism with $g(\mathbb P^1)$  not everywhere tangent to $\F$, and let $k$ be the number of
	summands of $g^*  T \F$ having strictly positive degree. If $x \in X$ is a general point, then
	there exists a quasi-projective variety $V_x$ of dimension at least $k$ passing through $x$ and contained in the
	leaf of $\F$ through $x$.
\end{prop}
\begin{proof}
	 If $L$ is the
	leaf of $\F_{tang}$ through $[g]$ and $\varphi  :  \mathbb P^1 \times L   \to    X$
	is the restriction to $\mathbb P^1 \times L$ of the evaluation  morphism,  then the differential
	\[
	d \varphi : T \mathbb P^1 \times T L \longrightarrow \varphi^* (  TX)
	\]
	of $\varphi$ at a point $(z,[h]) \in \mathbb P^1 \times L $ is given by
	\[
	d \varphi(z,[h]) =   dh(z) + \phi(z,[h])
	\]
	where $\phi(z,[h]) : H^0(\mathbb P^1, h^* T\F) \to h^*TX \otimes \mathcal O_X / \mathfrak m_z \mathcal O_X$ is the evaluation morphism,
	and $dh$ is the differential of $h$, see \cite[page 114]{kollar}.
	If $z_0 \in \mathbb P^1$ and $[g]$  are general enough, then the kernel of $d \varphi(z_0,[g])$ has dimension at least  $h^0(\mathbb P^1,g^*T \F \otimes \mathfrak m_{z_0})$,
	where $\mathfrak m_{z_0} \subset \mathcal O_{\mathbb P^1,z_0}$ denotes the maximal ideal of the local ring of $\mathbb P^1$ at $z_0$.
    Since the pair $(z_0,[g])$ is general and $g$ has image generically transverse to $\F$, we have that the set $A= \varphi^{-1}( \varphi(z_0,[g]) \cap \{ z_0\} \times L$ also has dimension  $h^0(\mathbb P^1,g^*T \F \otimes \mathfrak m_{z_0})$. Let $T$ be the projection of $A\subset \mathbb P^1 \times L$ to $L$.  Thus, we have an
	analytic family of morphisms $T$  contained in  $M \subset \Mor(\P^1,X)$, all of them mapping $z_0$ to  $p$.
    The description of $d \varphi$ given above implies that  for a general $y \in \mathbb P^1$, $y\neq z_0$,  the image under $\varphi$ of $\{y\} \times T$
	has dimension at least $k$, the number of non-negative summands of $g^*T \F \otimes \mathfrak m_{z_0} \cong g^*T \F \otimes \mathcal O_{\mathbb P^1}(-1)$.
	
	Let $\overline T$ be the Zariski closure of $T$ in $M$, and let $M_{p}\subset M$ be the set morphisms in $M$ mapping $z_0$ to $p=g(z_0)$.
    Clearly, $M_p$ is a closed subset of $M$, and  $\overline T$ is contained in $M_p$.
	Assume  $p$ is a smooth point of $\F$. We claim that $\overline T$ is tangent to $\F_{tang}$. To verify this claim let $H : (X,g(z_0)) \to (\mathbb C^q,0)$ be a germ of submersion defining $\F$ at $g(z_0)$,
    and consider the sets $\Sigma_k \subset M_p$ consisting of morphisms $f \in M_p$ such that $f^* H = g^* H \mod \mathfrak m_{z_0}^k$. Clearly, $\Sigma_k$
	is a closed algebraic subset of $M_p$ and, by design, the irreducible components of the intersection $\cap_k \Sigma_k$  are  tangent to $\F_{tang}$.
    Since  at least one of these irreducible components contains $T$, the claim follows.
	
	If $F : \mathbb P^1 \times M  \to X$
	is  the  evaluation morphism,
	then for  a general $y \in \mathbb P^1$, the image $F( \{y\} \times \overline T)$  will be an algebraic  set of dimension at least $\dim F( \{y\} \times T) \ge  k$ contained in the leaf
	of $\F$ through $g(y)$. The proposition follows.
\end{proof}


\begin{prop}\label{P:bendandbreak}
	Let $M \subset \Mor(\P^1,X)$ be an irreducible component containing free morphisms, and $[g] \in M $ be a general element. Suppose  $g^*  T \F$
	has at least one summand having strictly positive degree. If $x \in X$ is a general point, then
	there exists a rational curve through $x$, and contained in the
	leaf of $\F$ through $x$.
\end{prop}
\begin{proof}
	The proof  is similar to the one of \cite[Lemma 5.2, Lecture I]{MiPe}.
	If $g(\P^1)$ is tangent to $\F$, then there is nothing to prove. Otherwise, according to Proposition \ref{P:subfol},
	the existence of a positive summand in the decomposition of $g^* T\F$
	implies that we can algebraically  deform $g$ along $\F$ in such a way that a general  point $z_0 \in \P^1$
	is mapped to $g(z_0)=x$ along the deformation.
	More precisely, there exists a smooth quasiprojective curve $C^0 \subset M$ contained
	in a leaf of $\F_{tang}$, and  such that  every $[h] \in C^0$ maps $z_0$ to $x$.
	
	Let $C$ be a smooth projective curve containing $C^0$ as an open subset. The evaluation morphism
	$F: \P^1 \times C^0 \to X$ extends to a rational map $F : \P^1 \times C \dashrightarrow X$. Generically
	$F$ must have rank two, as otherwise the deformation would have to move points along the image $g(\P^1)$ of
	one of its member, and this is only possible if $g(\P^1)$ is tangent to $\F$. Notice also that
	$F^* \F$ is nothing but the foliation on $\P^1 \times C $ defined by the projection $\P^1 \times C
	\to \P^1$.
	
	Since the curve $C_0= \{ p_0\} \times C$ has self-intersection zero in $\P^1 \times C$, and $F$ has image of dimension two, there must exists an indeterminacy point of $F$ on $C_0$. By resolving the indeterminacies
	of $F$, we obtain a surface $S$ together with a morphism $G : S \to X$ fitting into the diagram
	\[
	\xymatrix{
		S \ar[d]^{\pi} \ar[drr]^{G} \\
		\P^1 \times C \ar@{-->}[rr]^{F} & & X \\
		\P^1\times C^0 \ar@{^{(}->}[u] \ar[urr]^{F}
	}
	\]
	where $\pi : S \to \P^1 \times C$ is a birational morphism. Moreover, there exists a curve $E \subset S$ contracted by $\pi$ into a point of $C_0$, whose image under $G$ is a rational curve on $X$ passing through $x$. Since
	the foliation $F^* \F$ is a smooth foliation on $\P^1 \times C$, every exceptional divisor
	of $\pi$ is also invariant by $(F \circ \pi)^* \F = G^* \F$. Therefore, $G(E)$ is the sought rational
	curve tangent to $\F$ passing through $x$.
\end{proof}

\subsection{Proof of Theorem \ref{T:transproj}}
Let $\F$ be a codimension one foliation with numerically trivial canonical bundle on a uniruled manifold. Let $f : \mathbb P^1 \to X$ be a general free morphism belonging to a fixed irreducible component $M$ of $\Mor(\mathbb P^1,X)$. If $f(\mathbb P^1)$ is contained in a leaf of $\F$, then $\F$ is uniruled. But this contradicts Corollary \ref{C:uniruled}, since we are assuming $\F$ has canonical singularities.

We can therefore suppose that $f$ is generically transverse to $\F$. Since $K_\F$ is numerically trivial, we obtain that $f^* T\F$ is either trivial, or has a non-trivial positive summand. If $f^* T\F$  has a non trivial positive summand for any general free morphism, then Proposition \ref{P:bendandbreak} implies $\F$ is uniruled. As before we arrive at a contradiction with Corollary \ref{C:uniruled}.

If $f^*T\F$ is trivial, then the foliation $\F_{tang}$ defined on $M$ has dimension equal to $h^0(\mathbb P^1, f^* T\F) = \dim X-1$. Let $\overline L$
be the Zariski closure of a general leaf of $\F_{tang}$. Notice that the restriction of the evaluation morphism $F:\mathbb P^1 \times M \to X$  to $Y = \mathbb P^1 \times \overline L \subset \mathbb P^1 \times M$ dominates $X$.  If we consider the restriction of $F$ to $Y$, then we are in position to apply Theorem \ref{T:6.4} in order to deduce that $(F^*\F)_{|Y}$ is transversely projective. To conclude the proof, we
apply  Lemma \ref{L:empurra}. \qed


\section{Reduction to positive characteristic}\label{S:modp}
Let $\F$ be a foliation defined on a complex projective manifold $X$. The variety $X$ and the subsheaf $T \F \subset TX$ can be both
viewed as objects defined over a ring $R$ of characteristic zero finitely generated over $\mathbb Z$.  If
$\mathfrak p \subset R$ is  a  maximal ideal, then $R / \mathfrak p$ is  a finite field $k$ of characteristic $p>0$.  The reduction
modulo $\mathfrak p$ of $\F$ is the {\it foliation} $\F_{\p}$ determined by the  subsheaf $T\F_\p  = T\F \otimes_R k$ of the tangent sheaf of the projective variety
$X_\p = X \otimes_R k$. In simple terms, we are just reducing modulo $\mathfrak p$   the equations (which have coefficients in $R$) defining $X$ and   $\F$. For more
on the reduction modulo $p$ see \cite[Chapter 1, \S2.5]{MiPe}.

Here, we will use reduction modulo $\p$ to  find invariant hypersurfaces and integrating factors
for   complex foliations with semi-stable tangent sheaves and  numerically trivial canonical bundle. We will implicitly make use of the following
result.

\begin{prop}
	Let $\F$ be a foliation on a polarized projective manifold $(X,H)$ defined over a finitely generated $\Z$-algebra $R \subset \mathbb C$. If there are  integers
	$M,m$,   and a Zariski dense  set of maximal primes $\mathscr P \subset \mathrm{Spec}(R)$ such that $\F_\p$ has an invariant subvariety of dimension $m$ and  degree
	at most $M$ for every $\p \in \mathscr P$, then $\F$ has an invariant subvariety of dimension $m$ and degree at most $M$.
\end{prop}
\begin{proof}
	For a fixed Hilbert polynomial $\chi$, the subschemes of $X$ invariant by $\F$ with Hilbert polynomial
	$\chi$ form a closed subscheme $Hilb_{\chi}(X,\F)$ of $Hilb_{\chi}(X)$, see \cite[Proposition 2.1]{CP}. Moreover, its formation commutes with base
	change. Thus $Hilb_{\chi}(X,\F)$ is non-empty if, and only if, $Hilb_{\chi}(X_\p,\F_\p)$ is non-empty for a Zariski dense set of primes $\p$,
	see for instance \cite[Lecture I, Proposition 2.6]{MiPe}.
	To conclude, it suffices to remind that irreducible reduced subvarieties of $X_\p$ of
	bounded degree have bounded Hilbert polynomial, independently of $\p$.
\end{proof}

If $v$ is vector field on a smooth algebraic variety  of positive characteristic, then its $p$-th power is also a vector
field, since it satisfies Leibniz's rule:
\[
v^p(f \cdot g) = \sum_{i=0}^p \binom{p}{i} v^{i}(f) v^{p-i}(g) = f v^p(g) + v^p(f) g \mod p \, .
\]
A foliation $\F$ on a smooth algebraic variety  $X$ defined over a field of characteristic $p>0$ is said to be  $p$-closed
if, and only if, for every local section $v$ of $T\F$ its $p$-th power $v^p$ is also a local section of $T\F$.

The $p$-closed foliations of codimension $q$ are precisely those that can be defined by $q$ rational functions $f_1, \ldots, f_q$ in the sense
that $df_1 \wedge \ldots \wedge df_q$ is a non-zero rational section of $\det N^* \F$ seen as a subsheaf of $\Omega^q_X$. Indeed, if $\F$ is a $p$-closed
foliation of codimension  $q$, then \cite[Lecture III, 1.10]{MiPe}  implies that at a general point of $X$ there are local coordinates in which $\F$
is defined by $dx_1 \wedge \cdots \wedge dx_q$. Reciprocally, if $\F$ is defined $df_1 \wedge \ldots \wedge df_q \neq0$, then for every rational vector field $v$
satisfying $i_v df_1 \wedge \ldots \wedge df_q =0$,
we have that
\[
i_{v^p} (df_1 \wedge \ldots \wedge df_q) = \sum_{i=1}^q (-1)^{i+1} v^p(f_i) \cdot df_1 \wedge \cdots \wedge \widehat{df_i} \wedge \cdots \wedge df_q = 0 \, .
\]
This illustrates what is perhaps the
most astonishing contrast between foliations in  positive/zero characteristic: the easiness/toughness to decide whether or not $\F$ has first integrals.

If   $\F$ is a foliation on a projective manifold defined over a finitely generated $\Z$-algebra $R \subset \C$
then  the behavior of $X_\p$ and $\F_\p$ may vary wildly when  $\p$ varies among the maximal
primes of $R$. Thus,
in order to have some hope to read properties of $\F$ on its reductions
modulo $\p$, one has to discard the {\it bad primes}. When a foliation $\F$ on a complex projective manifold has  $\p$-closed reduction modulo $\p$ for every maximal prime ideal $\p$ lying in a nonempty  open subset $U \subset \mathrm{Spec}(R)$, then we will simply say that $\F$ is $p$-closed.

\subsection{Integrating factors in positive characteristic}
In this section, we collect some  results from  \cite[Section 6]{Croco2}
which will be essential in what follows.

\begin{lemma}\label{L:pos}
	Let $X$ be a  smooth affine variety of dimension $n$ defined over a field of arbitrary characteristic. If $\omega$
	is an integrable  $1$-form which is  non zero at a closed  point $x \in X$, then there exists $n-1$
	regular vector fields $v_1, \ldots, v_{n-1}$ at an affine  neighborhood of $x$  such that
	\begin{enumerate}
		\item $v_1 \wedge \cdots \wedge v_{n-1}(x) \neq 0$;
		\item $[v_i, v_j] = 0$ for every $i, j \in \{1, \ldots, n-1\}$;
		\item $i_{v_i}\omega = 0$ for every $i \in \{1, \ldots, n-1\}$.
	\end{enumerate}
\end{lemma}
\begin{proof}
    This is lemma 6.1 from \cite{Croco2}.
\end{proof}

The underlying idea of the proof of  the next result is that $p$-th powers of vector fields
tangent to a integrable $1$-form give rise to infinitesimal automorphisms, and these allow us to find integrating factors.
The proof presented below is borrowed from  \cite[proof of Theorem 6.2]{Croco2}. We have chosen
to present it here, since this result is pivotal  in the proof of Theorem \ref{T:modulop}.

\begin{prop}\label{P:invp}
	Let $X$ be a smooth  variety defined over a field $k$
	of characteristic $p >0$, and $\omega$ be a rational $1$-form on $X$.
	If $\omega$ is integrable and there exists a rational vector field  $\xi$ such that
	\begin{enumerate}
		\item $i_{\xi} \omega = 0$; and
		\item $F=\omega(\xi^p) \neq 0$
	\end{enumerate}
	then the $1$-form $F^{-1}\cdot{\omega}$ is closed.
\end{prop}
\begin{proof}
	Let $n$ be the dimension of $X$ and $v_1, \ldots, v_{n-1}$ be the rational vector fields
	given by Lemma \ref{L:pos}.
	Thus $\xi = \sum_{j=1}^{n-1} a_{ij} v_j$ for suitable rational functions $a_{ij}$.
	By a formula of Jacobson \cite[page 187]{Jacobson}, we can write
	$
	\xi^p = \sum_{j=1}^{n-1} a_{ij}^p v_j^p + P(a_{i,1}v_1,\ldots, a_{i,n-1}v_{n-1})
	$
	with $P$ being a Lie polynomial. Since $[v_i,v_j]=0$ it follows that
	\[
	\xi^p = \sum_{j=1}^{n-1} a_{ij}^p v_j^p \, , \mod < v_1, \ldots, v_{n-1} > \, .
	\]
	As we are interested in contracting $\xi^p$ with $\omega$, we will replace $\xi^p $ by   $\zeta= \sum_{j=1}^{n-1} a_{ij}^p v_j^p$. Notice that
	$[\zeta,v_j] = 0$   for  $j \in \{ 1, \ldots, n-1\}$.
	
	Set $\alpha = \frac{\omega}{\omega(\zeta)}$.  The integrability of $\omega$  together with $ \alpha(\xi^p)= \alpha(\zeta)=1$ implies
	\[
	0 = i_{\zeta} (\alpha \wedge d \alpha) =   d \alpha - \alpha \wedge i_{\zeta} d \alpha \, .
	\]
	Hence to prove that  $\alpha$ is closed, it suffices to verify that the $1$-form $i_{\zeta} d \alpha$ is zero.
	As the vector fields $v_1, \ldots, v_{n-1}, \zeta$ commute, then for every vector field $v$ in the previous list we have
	\[
	(i_{\zeta} d \alpha) (v) =  \alpha([\zeta,v])  - \zeta(  \alpha(v) ) + v(\alpha(\zeta) ) = 0 \, .
	\]
	This ensures that $i_\zeta d \alpha=0$, and consequently $d \alpha = 0$. The proposition follows.
\end{proof}

\begin{cor}\label{C:colap}
	Continuation of Proposition \ref{P:invp}:
	if $\tilde{\xi }$ is
	another  rational vector field satisfying (1) and (2), then the rational functions $F=\omega(\xi^p)$
	and $\tilde F=\omega(\tilde{\xi}^p)$ differ by the multiplication of a $p$-th power of a rational function, i.e.,
	$F = H^p \tilde F$, for some rational function $H$. In particular, the identity $\frac{dF}{F} = \frac{d \tilde F}{ \tilde F}$ holds true.
\end{cor}
\begin{proof}
	According to Proposition \ref{P:invp}, both   $F^{-1} \omega$ and $\tilde{F}^{-1} \omega$ are closed $1$-forms. Therefore
	$d ( F^{-1} \tilde{F}) \wedge \omega = 0$. Since the foliation defined by $\omega$ is not $p$-closed, it follows
	that  $d ( F^{-1} \tilde{F})=0$. Hence  $F = H^p  \tilde F$ for a suitable rational function $H$.
\end{proof}

\subsection{Lifting integrating factors}
We will now proceed to prove the main result of Section \ref{S:modp}.
\begin{thm}\label{T:modulop} Let $(X,H)$ be a polarized projective complex manifold, and $\F$ be a semi-stable foliation of codimension one on $X$.
	If $\KF \cdot H^{n-1} =0$, then at least one of the following assertions holds true
	\begin{enumerate}
		\item the foliation  $\F$  is $p$-closed;
		\item  $\F$ is induced by a closed rational $1$-form with coefficients in a flat line bundle, without divisorial components
		in its zero set.
	\end{enumerate}
\end{thm}
\begin{proof}
	Let $R \subset \mathbb C$ be a finitely generated $\mathbb Z$-algebra such that everything in sight is defined over it.
    Suppose that the set  of maximal primes $\mathcal P \subset \mathrm{Spec}(R)$ for which $\F_\p$ -- the reduction
    modulo $\p$ of $\F$  -- is not $p$-closed is Zariski dense, and fix $\p\in\mathcal P$.
	
	To  raise germs of vector fields $v$ in $T \F_\p$ to theirs $p$-th powers  provides  a non-zero global section  $S_\p$ of
	$$\Hom_{{\mathcal O}_X}(F^*T\F_\p,   N\F_\p)={(F^*T\F_\p)}^* \otimes  N\F_\p$$
	where $F$ is the absolute Frobenius.
	
	Let us explicitly describe $S_\p$ at
	sufficiently small Zariski open subsets $U_i$ disjoint from the singular set of $\F_\p$.
	Let  $v_{1,i},....,v_{n-1,i}$ be the $n-1$ vector fields satisfying the conclusion of  Lemma \ref{L:pos}, and
	such that $v_{1,i}\wedge...\wedge v_{n-1,i}$ does not vanish on $U_i$.
	We can also assume that $\F_\p$ is defined on the same domain by a $1$-form $\omega_i$ without
	divisorial components in its singular set. Take another open set $U_j$ with the same properties. On overlapping charts, we have
	$$\begin{pmatrix}v_{1,i}\\
	\vdots\\
	v_{n-1,i}\end{pmatrix}=M_{ij}\begin{pmatrix}v_{1,j}\\
	\vdots\\
	v_{n-1,j}\end{pmatrix}$$
	where the matrix cocycle $\{ M_{ij} \}$ represents the cotangent bundle $T^*{\F_\p}$ of the foliation outside $\mathrm{sing}(\F_\p)$.
	
	As a consequence, using Jacobson's formula \cite{Jacobson}, we obtain
	\[
        \begin{pmatrix}
	       {v_{1,i}^p}\\
    	   \vdots\\
    	   v_{n-1,i}^p
        \end{pmatrix} = N_{ij}
        \begin{pmatrix}
            v_{1,j}^p\\
            \vdots\\
            v_{n-1,j}^p
        \end{pmatrix} \mod\ T \F_\p
    \]
	where  the matrix $N_{ij}$ is obtained from $M_{ij}$ by replacing each entry by its $p^{th}$ power. If we set
	$s_{k,i}=\omega_i(v_{k,i}^p)$, then we gain the following equality
	\[
        \begin{pmatrix}
            {s_{1,i}}\\
    	    \vdots\\
    	     {s_{n-1,i}}\end{pmatrix}=g_{ij}N_{ij}\begin{pmatrix}{s_{1,j}}\\
        	\vdots\\
        	{s_{n-1,j}}
        \end{pmatrix}
    \]
	where $g_{ij}$ is the cocycle representing the normal bundle of $\F_\p$.
    The collection of vectors $\{ (s_{1,i} s_{2,i} \ldots s_{n-1,i})^T \}$ represents $S_\p$
	on $X_\p - \sing(\F_\p)$.


	Let $D_\p$ be the zero divisor of the section $S_\p$. Over $U_i$, $D_\p$ is defined by the codimension one  components of the common zeros of $s_{1,i}, \ldots, s_{n-1,i}$.
	Since   $\F_\p$ is not $p$-closed, there is at least one among these functions which does not vanish identically.
    Choose one for each open subset $U_i$  and denote it by $s_i$. Corollary \ref{C:colap} guarantees that the zero divisor of two different
    choices will differ by an element in $p \cdot \mathrm{Div}({U_i})$. It also implies that over nonempty intersections
    $U_i \cap U_j$, we have $s_i = g_{ij} h_{ij}^p s_j$ for some rational function $h_{ij}$ in $U_i \cap U_j$. Therefore the
    rational $1$-forms $\frac{ds_i}{s_i}$ do not depend on the choices of the rational functions $s_i$ and they satisfy
	\[
        \frac{ds_{i}}{s_{i}} -  \frac{ds_{j}}{s_{j}}= \frac{dg_{ij}} {g_{ij}}\,  .
    \]
	Notice also that the polar set of $ds_i/s_i$ coincides with the irreducible components of ${D_{\p}}_{|U_i}$ which have multiplicity
	relatively prime to $p$.
		
	If we write $g_{ij}= g_i/g_j$ as a quotient of rational functions on $X$, then we can define on $X_\p$  a closed rational $1$-form
	with simple poles $\eta_\p$ by setting
	\[
	   (\eta_{\p})_{|U_i} =  \frac{ds_i}{s_i} - \frac{dg_i}{g_i}  \,
	\]
	where we still denote by $g_i$ the reduction modulo $\p$ of the rational functions $g_i$.
		
	If $C_\p$ is  an  irreducible curve on $X_\p$ not contained in the polar set of $\eta_\p$, then
	the restriction of $\eta_{\p}$ to $C_{\p}$ is  a rational $1$-form with  sum
	of residues equal to $(D_{\p} - N\F_\p) \cdot C_\p  \mod p$. The residue formula implies the equality
	\begin{equation}\label{E:nfmodp}
		D_\p  \cdot C_\p =  N \F_\p \cdot C_\p \mod p \, \, .
	\end{equation}
		
	If we set $\omega = \omega_i/g_i$, then $\omega$ is a well-defined rational $1$-form on $X$. Moreover, Proposition \ref{P:invp} implies that the identity
	$
	   d\omega = \eta_\p\wedge\omega
	$
	holds true on $X_\p$. If we were in characteristic zero, then the closed $1$-form $\eta_\p$ would be the sought integrating factor.
	
	Up to this point, we have not used the hypothesis on $\KF$.
	In order to explore it and  obtain further restrictions on  $D_\p$, we will use  the following result by Shepherd-Barron, \cite[Corollary $2^p$]{SB} and \cite{kol}.
	
	\begin{lemma} \label{L:sb} (char $p$)
        Suppose that $\mathcal E$ is a semi-stable vector bundle of rank $r$ over a curve $C$ of genus $g$. Consider $F^* \mathcal E= \tilde {\mathcal E}$, the pull-back of $\mathcal E$ under the absolute Frobenius. Then there exists $M=M(r,g)>0$ independent of p such that
		\[
             \mu_{max}(\tilde{\mathcal E})- \mu_{min}(\tilde{\mathcal E})\leq M.
        \]
	\end{lemma}
	
	Now, return to the original foliation $\F$ on the complex manifold $X$. Consider a general complete intersection curve $C$
	cut out by elements of $|mH|$ ( $m \gg 0$ ) for which the  $T \F_{|C}$ is semi-stable. Notice that this semi-stability is preserved under specialization $\mod \p$ for almost every $\p$.
	
	Restricting $S_\p$ to $C_\p$ and cleaning up its zero divisor, we get a section of
	$$\Hom_{{\mathcal O}_{C_\p}}({F^*\F_\p}_{|C_\p},{N\F_\p}_{|C_\p}\otimes\mathcal O_{C_\p}(-D_\p)).$$
	Since $\Hom_{{\mathcal O}_X}(\mathcal A,\mathcal B)=0$ whenever $\mu_{min}\mathcal A>\mu_{max}\mathcal B$, we deduce that
	\[
    	\mu_{min}( F^*{\F_\p}_{|{C_\p}})  \le   N{\F_\p} \cdot {C_\p}  - D_\p \cdot C_\p\, .
	\]
	Lemma \ref{L:sb} and the fact that $\mu_{max}({F^*\F_{|C_\p}})\ge0$ imply
	\[
	   D_\p \cdot C_\p  \le M +  N\F_\p \cdot C_\p
	\]
	with $M$ uniform in $\p$. For $p \gg0$,    this last inequality combined with (\ref{E:nfmodp}) implies    $D_\p \cdot C_\p = N \F_{\p} \cdot C_{\p} = N\F \cdot C$.
    Consequently,  the degree of $D_\p$ is uniformly bounded. In particular,   the polar locus of $\eta_\p + \frac{dg_i}{g_i}$ on $U_i$ coincides with the support of  ${D_{\p}}_{|U_i}$ and its residues are positive integers uniformly bounded with respect to $p$ (indeed, they coincide with the multiplicity of irreducible components of $D_\p$).
    Thus, there exist on $X$ 
	a closed rational  $1$-form $\eta$ with simple poles such that 
	\[
    	d \omega =  \eta \wedge\omega, .
	\]
 and such that the residues of $ \eta + \frac{dg_i}{g_i}$ lie in $\mathbb Z_{>0}$. 
 We conclude that  the $1$-form
	\[
	   \frac{ \omega} { \exp \int \eta }
	\]
	is a closed rational $1$-form on $X$ with zero set of codimension at least two  and coefficients in a flat line bundle defining
	the foliation $\F$.
\end{proof}

\subsection{Singularities of  $p$-closed foliations}
McQuillan  observed in \cite[Proposition II.1.3]{McQ} that  isolated singularities of $p$-closed foliations of dimension one  with non-nilpotent linear
part are fairly special.

\begin{lemma}\label{L:singcharp}
	Let $\F$ be a $p$-closed foliation by curves on a projective manifold $X$. If $x \in \mathrm{sing}(\F)$ is an
	isolated singularity with non-nilpotent linear part, then there exist formal  coordinates at $x$ where  $\F$
	is generated by  the linear vector field
	\[
	v = \sum_{i=1}^n \lambda_i x_i \frac{\partial} {\partial x_i}
	\]
	where $\lambda_1, \ldots, \lambda_n$ are non-zero integers.
\end{lemma}
\begin{proof} In the terminology of \cite{McQ}, we have that an isolated singularity is log canonical if, and only if, its linear part is not nilpotent \cite[Fact I.1.8]{McQ}.
	The only case in our statement  not covered by \cite[Proposition II.1.3]{McQ} is when the singularity is log canonical but not canonical.
	According to \cite[Fact I.1.9]{McQ}, this implies that the vector field is linearizable  and all its quotients of eigenvalues are positive rational numbers.
\end{proof}

\begin{cor}\label{C:prime}
Let $\F$ be a foliation on a projective manifold $X$ with numerically trivial canonical bundle.
If the singularities of $\F$ are canonical, then one of the following assertions hold true.
\begin{enumerate}
\item The foliation $\F$ is defined by a closed rational $1$-form with coefficients in a flat line bundle and
without divisorial components in its zero set.
\item At a general point of every irreducible component of codimension two of $\sing(\F)$, the foliation
admits a holomorphic first integral of the form $x^py^q$ where $p,q$ are positive integers.
\end{enumerate}
\end{cor}
\begin{proof}
    If $\F$ is not $p$-closed, then the result follows from Theorem \ref{T:modulop}.
    If instead $\F$ is p-closed, then the same holds true for the restriction of $\F$
    to any projective surface $S\subset X$. Lemma \ref{L:singcharp} combined with Proposition \ref{P:codim2} implies the result.
\end{proof}

Although not strictly necessary for what follows, Theorem \ref{T:truquesujo} plays an essential
role in the classification of codimension one foliations with trivial canonical bundle on Fano $3$-folds with
Picard number one carried out in  \cite{croco3b}. In contrast with the corollary above,  no assumptions are
made on the nature of the singularities of the foliation.

\begin{thm}\label{T:truquesujo}
	Let $(X,H)$ be a polarized complex projective  manifold and $\F$ be a codimension one semi-stable foliation on $X$ with  numerically trivial
	canonical bundle. Suppose  $c_1(TX)^2 \cdot H^{n-2} >0$.
	If $\F$ is $p$-closed, then
	\begin{enumerate}
		\item  $\F$ is a  rationally connected foliation, i.e., the general leaf of $\F$ is a  rationally connected
		algebraic variety; or
		\item $T\F$ is not stable  and there exists  a rationally  connected  foliation $\mathcal H$ tangent to $\F$ and with $K_{\mathcal H} \cdot H^{n-1}=0$.
	\end{enumerate}
\end{thm}
\begin{proof}
	As $c_1(T\F)=0$,  we have that $c_1(TX) = c_1(N \F)$. Thus
	$c_1(TX)^2 \cdot H^{n-2} = c_1(N\F)^2 \cdot H^{n-2} >0$ and Baum-Bott index theorem \cite{BB}
	implies the existence of a codimension two component $S$ of the singular set of  $\F$ which has
	positive Baum-Bott index.
		
	Take a general surface $\Sigma \subset X$ intersecting $S$ transversally. Since $\dim \Sigma=2$
	all the singularities of $\F_{|\Sigma}$ are isolated.
	As $p$-closedness is preserved by  restrictions to subvarieties, it follows from Lemma \ref{L:singcharp} that a  singularity
	of $\F_{|\Sigma}$ on  $S\cap \Sigma$ either   has nilpotent linear part, or  is  formally linearizable with rational
	quotient of  eigenvalues. Moreover, since the Baum-Bott index of $S$ is positive, in the latter case the quotient
	of eigenvalues must be positive. In both cases $S$ is a non-canonical center, i.e., there exists a composition of blow-ups on smooth centers  $\pi:Y \to X$  such that the canonical bundle of $\mathcal G= \pi^* \F$ is
	of the form
	\begin{equation*}
	   K_{\mathcal G} = \pi^* \KF  - E - D.
	\end{equation*}
	Here $E$ is an effective divisor supported on an irreducible
	hypersurface   such that $\pi(|E|) = S$ and $D$ is a divisor
	(not necessarily effective) such that $\pi(|D|)  \subset  \mathrm{sing}(S)$. In particular $\pi(|D|)$ has codimension
	at least three.
	
	Let $A$ be an ample line bundle on $Y$, and $H_{\varepsilon}$ equal to $\pi^*H + \varepsilon A$.
    Let  $\varepsilon_0 >0$ be  such that $K \G \cdot H_{\varepsilon}^{n-1} <0$ for any positive $\varepsilon \le \varepsilon_0$.
	If  $T\mathcal G$ is $H_{\varepsilon}$-semi-stable for some positive $\varepsilon \le \varepsilon_0$,
	then  Corollary \ref{C:MBM} implies that the leaves of $\mathcal G$ are rationally connected varieties, and we can conclude.
	If not, then for any positive $\varepsilon\le \varepsilon_0$,
	the maximal destabilizing foliation of $\mathcal G$, which we will denote by $\mathcal H_{\varepsilon}$,  satisfies
	\begin{equation*}
    	\mu_{\varepsilon} (T \mathcal H_{\varepsilon}) > \mu_{\varepsilon} ( T \mathcal G) \,
	\end{equation*}
	where the slope $\mu_{\varepsilon}$ is computed as a function of $A$ and $\varepsilon$.
	
	A priori, as $\varepsilon$ goes to zero the maximal destabilizing foliation $\mathcal H_{\varepsilon}$ could vary, but the proof of \cite[Lemma 3.3.3]{Neumann}
	shows that this cannot happen. More precisely, for $\varepsilon >0$ sufficiently small, the maximal
    destabilizing foliations $\mathcal H_{\varepsilon}$  will
	be all equal to a fixed foliation $\mathcal H$.
	Since $\mu_{H_{\varepsilon_i}}(T \mathcal H) > \mu_{H_{\varepsilon_i}}(T\mathcal G)>0$,  Corollary \ref{C:MBM}
    implies that the general leaf of $\mathcal H$ is rationally connected.
	
	To conclude, notice that on the one hand  $\mu_{\varepsilon_k}(T \mathcal H) > \mu_{\varepsilon_k}(T \mathcal G) > 0$ implies that $\mu_{\pi^*H}(T \mathcal H) = \mu_H(T \pi_*  \mathcal H) \ge 0$.
	On the other hand, the $H$-semistability of $T \F = T \pi_* \mathcal G$ implies  $\mu_H(T \pi_*  \mathcal H) \le 0$. It follows
	that $\mu_H(T \pi_*  \mathcal H) = 0$. Consequently  $T \F$  is   semi-stable but not stable, and $\pi_* \mathcal H$ is the sought foliation tangent to $\F$ with
	rationally connected general leaf.
 \end{proof}

\section{Closed rational differential forms }\label{S:structure}

In this section we will  prove the following result.

\begin{thm}\label{T:F}
    Let $\F$ be a codimension one foliation with numerically trivial canonical bundle on a  projective manifold $X$.  If the singularities of $\F$ are canonical, then
    $\F$ is defined by a closed rational $1$-form without divisorial components on its zero set after a finite \'etale covering, or $\F$ is an isotrivial fibration.  Otherwise, $\F$ is uniruled.
\end{thm}

Let us briefly recall what we already know.
\begin{enumerate}
\item If $\F$ has non canonical singularities, then $\F$ is uniruled according to Corollary \ref{C:uniruled}.
\item If $\KX$ is pseudo-effective, then $\F$ is smooth according to Theorem \ref{T:uniruled iff non psef}.
\item If $\F$ is not $p$-closed, then Theorem \ref{T:F} follows from  item (2) of Theorem \ref{T:modulop} in the situation where the flat line bundle involved is torsion. This is indeed the case and will be proved in \ref{SS:flator}.
\item\label{I:5} If $\F$ is $p$-closed, then at a general point of a codimension two irreducible component
of $\sing(\F)$, the foliation  $\F$ admits a local holomorphic first integral according to Corollary \ref{C:prime}.
\item If $X$ is uniruled, then $\F$ is transversely projective according to Theorem \ref{T:transproj}.
\end{enumerate}
To conclude the proof of Theorem \ref{T:F}, we will show that a foliation $\F$
with numerically trivial canonical bundle, with canonical singularities satisfying the conclusion of \ref{I:5},
and admitting a  transversely projective structure, necessarily satisfies the assumptions of Lemma \ref{L:logdivbis}.
Once this is done, Theorem \ref{T:F} follows since Lemma \ref{L:logdivbis} implies  $H^1(X,N^* \F)\neq 0 $, and
we can apply Corollary \ref{C:smooth} to conclude that $\F$ is smooth. Theorem \ref{T:F} will then  follow from the lemma below.

\begin{lemma}
Theorem \ref{T:F} holds true for smooth foliations.
\end{lemma}
\begin{proof}
The proof is a consequence of the classification of smooth codimension one
 foliations with trivial canonical class \cite{Touzet} recalled in Section \ref{S:oldresults}.
To prove Theorem \ref{T:F}, it suffices to show  that, after a finite \'etale covering,
such foliations are defined by a closed rational $1$-form without codimension one zeros. This is clear
for the foliations described in (1) and (2) of Section \ref{S:oldresults}.

Let us now  verify the result for the foliations transverse
to a fibration by rational curves (case (3) of Section \ref{S:oldresults}). It follows
from the classification of projective manifolds with trivial canonical class that they
have virtually abelian fundamental. Therefore, after a finite \'etale  covering,
we can assume that the lifting of closed paths in $Y$ along leaves of $\F$ will induce a representation with values
$(\mathbb C^*,\cdot)\subset \Aut(\mathbb P^1)$  or $(\mathbb C,+) \subset \Aut(\mathbb P^1)$. In both
situations, this holonomy representation preserves a rational $1$-form $\eta$ on $\mathbb P^1$ without zeros.
The invariance of $\eta$ by the holonomy representation allows us to saturate it by the foliation
in order to obtain a closed rational $1$-form without zeros defining $\F$.
\end{proof}

\subsection{Transversely projective structure, Schwarzian derivative, and invariant divisors}\label{S:Schwarz}
Suppose $\F$ is a codimension one foliation on a projective manifold $X$ with numerically trivial canonical bundle.
We will also assume that, at the general point of every irreducible component of the singular set of $\F$ having codimension two,  the foliation is defined by
\[
p x dy + q y dx
\]
with $p,q$ relatively prime positive integers.
We will  make use of the transversely projective structure given by Theorem \ref{T:transproj} to produce a divisor
satisfying the hypothesis of Lemma \ref{L:logdiv}.
The degeneracy locus of the transversely projective structure is of the form $\sing(\F) \cup \Sigma$, where $\Sigma$ is a finite union of $\F$-invariant hypersurfaces, see for instance \cite[Lemma 2.2]{RH},
or \cite[Proposition 2.15]{Croco2}.  Outside this  set, the foliation is defined by local submersions with values in $\P^1$, and transition functions in $\Aut(\P^1)$. We emphasize that the transverse structure gives distinguished first integrals for the foliation $\F$ outside the degeneracy locus of the projective structure. We will denote the sheaf of such distinguished first
integrals by $\mathcal I$.

Consider a regular point $p \in X - \sing(\F)$ where the foliation is locally given by a submersion $z$, $z(p)=0$. We can select an open neighborhood $U$ of $p$ and a section $f$ (possibly multi-valued) of $\mathcal I$ (which depends only on the $z$ variable) such that the Schwarzian derivative of $f$ with respect to $z$,
$\{f,z\}$,  is a well defined meromorphic function on the whole open set $U$. Hence, we can expand the Schwarzian derivative of $f$ with respect to $z$  as
\[
    \{ f, z\} =\sum_{i\geq i_0}a_ i z^i
\]
with $i_0 \in \mathbb Z$ and $a_{i_0} \neq 0$, unless $\{f,z\}$ vanishes identically.
The following facts can be easily verified.
\begin{enumerate}
	\item The first integral $f$ is a submersion if, and only if, $i_0\geq 0$.  In particular, if $i_0<0$, then the local
	      invariant hypersurface   $\{z=0\}$ actually belongs to an algebraic hypersurface in $\Sigma$.
	\item If $i_0 \leq -1$, then $i_0$ is independent of the choice of the local coordinate $z$. Consequently,
    	  $i_0$ is constant along the irreducible hypersurfaces, in $\Sigma$. If $H$ is one of such hypersurfaces, then
      	  we will denote by $i_0(H)$ the value of $i_0$ along it. Moreover, if $i_0\geq-2$, then  the coefficient of
          $\frac{1}{z^2}$  is independent of the coordinate  and we define  $a(H)=a_{-2}$.
	\item there exists $\varphi \in \Aut(\mathbb P^1)$ such that the function  $\varphi \circ f(z) - \log z$  is holomorphic if, and only if,  $i_0=-2$ and $a_{-2}=\frac{1}{2}$.
\end{enumerate}

We will say that $H$ is an irregular singularity of the projective structure if, and only if, $i_0(H)<-2$. Otherwise, if $i_0(H)\in\{-2,-1\}$,  we will say that $H$ is a regular singularity.

\subsubsection{Passing through corners}

Let  $\omega=pydx +qx dy$ be a germ of  $1$-form at the origin  of $ {\mathbb C}^n$ with $p, q$ relatively prime positive integers.
Suppose that the foliation $\F$ induced by $\omega$ is endowed with a projective structure.  Let $f$  be a multi-valued section of $\mathcal I$  defined on (the universal cover of) the complement of $\{xy=0\}$.
Let $r=\frac{q}{p}$. Set $i_x=i_0(\{x=0\})$ and  $i_y=i_0(\{y=0\})$.
On the transversals $\{y=1\}$ and $\{x=1\}$, we get respectively
\[
    \{f,x\} =\sum_{i\geq i_x} a_ix^i \quad \text{ and } \quad \{f,y\}=\sum_{i\geq i_y}b_iy^i \,.
\]

\begin{lemma}\label{L:facts}
    Notation as above. The following assertions hold true.
	\begin{enumerate}
		\item  If the singularity on $\{x=0\}$ is irregular, i.e.  $i_x<-2$, then  $i_y= r(i_x+2)-2$ and $b_{i_y}=a_{i_x}r^2$. Therefore,  $i_y<-2$ and the singularity  on $\{y=0\}$ is also irregular.
		\item If the singularity on $\{x=0\}$ is regular with $i_x=-2$, then $i_y\ge-2$ and
		$b_{-2}= r^2(a_{i_x}-\frac{1}{2})+\frac{1}{2}$.
		\item If $i_x \ge -1$ and $r\neq 1$, then $i_y=-2$.
	\end{enumerate}
\end{lemma}
\begin{proof}
The   restrictions of $f$ to the transversals $\{y=1\}$ and $\{x=1\}$ are related by the so called Dulac's transform which is  a (multi-valued)
holonomy transformation between the two transversals. It is  explicitly given by $x=h(y)=y^r$.

The  composition rule for the Schwarzian derivative
\[
    \{f\circ h, z\} = \{f,h(z)\} {h'}^2(z)+\{h,z\}
\]
applied to  $x=h(y)=y^r$, together with the fact that
\[
    \{h,y\}= \frac{1-r^2}{ y^2} \,
\]
implies the lemma.
\end{proof}

It follows that the projective structure determines a logarithmic $1$-form $\eta$ with {\it canonical} residues on $\{xy=0\}$ satisfying $d\omega = \eta \wedge \omega$ as follows.
\begin{enumerate}
	\item In case of irregular singularities: $\eta= (-i_x-3)\frac{dx}{x} +(-i_y-3)\frac{dy}{y}$.
	\item In case of regular singularities: $\eta =({\vert 2a_{-2}-1\vert}^\frac{1}{2} -1)\frac{dx}{x}+({\vert 2b_{-2}-1\vert}^\frac{1}{2} -1)\frac{dy}{y}$.
\end{enumerate}

The square roots appearing in formula (2)  are the positive square roots of the corresponding absolute values.
In case the projective structure has regular singularities along one of the branches, say $\{ x=0\}$,
it may happen that the distinguished first integrals induced by the projective structure are holomorphic submersions  at a general point of  $\{ y=0\}$, e.g. the distinguished first integrals are of the form $x^py$ composed with elements of  $\Aut(\mathbb P^1)$. If this happens, then  $b_{-2}=0$ and the $1$-form
$\eta$ is holomorphic along $\{ y=0\}$. It also may happen that the distinguished first integrals are holomorphic at  general points of both branches. In this case, the distinguished first integrals are defined at the complement of $\{ x=y=0\}$, thanks to the  simple-connectedness of $(\mathbb C^n,0) - \{ x=y=0\}$, and therefore extend
through the singular set by Hartogs. When this happens, then $p=q=1$, $a_{-2}=b_{-2}=0$, and $\eta$ is equal to zero.

In any  case the   residues of $\eta$ are  real. Both residues are equal to $-1$, or  both
residues are  strictly greater than $-1$.

\subsection{Proof of Theorem \ref{T:F} }

Let $\F$ be a transversely projective foliation which is of the form $pxdx +qydx$ ($p, q$ relatively prime positive integers)
at the general point of every codimension two irreducible component $S$ of $\sing(\F)$.  We will now
construct a divisor $D$ with support  on $\Sigma$ (the singular set of the transversely projective structure)  satisfying the hypothesis of Lemma \ref{L:logdivbis}.

Write $\Sigma = \Sigma_1 \cup \ldots \cup \Sigma_{\ell}$ as the union of its connected components.
Fix a connected component $\Sigma_j$ and pick a point $p\in \sing(\F) \cup \Sigma_j$
in an irreducible component  $H_p \subset \Sigma_j$.

Assume  that $i_0(H_p)<-2$. Then every other irreducible component  $H$ of $\Sigma_j$
must satisfy $i(H)<-2$ according to Lemma \ref{L:facts}.
Then we set
\[
    D_j=\sum_{H \subset \Sigma_j} (-i(H)-3)H .
\]
Notice that $D_j$ satisfies the hypothesis of Lemma \ref{L:logdivbis} in a neighborhood of $\Sigma_j$.

Assume that $i_0(H_p)=-2$ and $a(H_p)=\frac{1}{2}$. Lemma \ref{L:facts} implies  the same holds true for every irreducible component $H \subset \Sigma_j$. Fact (3) of Section \ref{S:Schwarz} implies that
over a general point of $\Sigma_j$,  we get a logarithmic first integral (induced by the projective structure) which gives rise to a well defined local section $\beta$ of $d\mathcal I$.
Indeed, these local sections are logarithmic $1$-forms with poles on $\Sigma_j$,
which are unique up to a multiplicative constant. Using Mayer-Vietoris sequence,   we deduce
the existence of a global logarithmic $1$-form $\beta_j$ on a neighborhood $V$ of $\Sigma_j$.  Moreover, one can choose $\beta_j$ with all its residues strictly positive.
We use the (global) $1$-form $\beta_j$ to change the (local) {\it canonical} $1$-form $\eta$
deduced from Lemma \ref{L:facts} to a  new $1$-form $\eta'= \eta + \beta_j$ having all its residues strictly greater than $-1$ and still satisfying $d \omega = \eta' \wedge \omega$.
In this case we set
\[
    D_j =\sum_{H \subset \Sigma_j} ( \mathrm{res}_H(\eta') )H .
\]

Assume that  $i_0 (H_p)=-2$ and $a(H_p) \neq\frac{1}{2}$. In this case, using again  Lemma \ref{L:facts}, if we set
\[
    D_j =\sum_{H \subset \Sigma_j} ({\vert 2a(H)-1\vert}^\frac{1}{2} -1)H,
\]
then $D_j$ satisfies the hypothesis of Lemma \ref{L:logdivbis} in a neighborhood of $\Sigma_j$.

If we sum up the divisors $D_j$ for all the connected components of  $\Sigma$,
we obtain a divisor $D$ satisfying the hypothesis of Lemma \ref{L:logdivbis},  and consequently we get the non-vanishing of  $H^1(X, N^* \F)$.
As we are assuming that $\F$ is not uniruled, Corollary \ref{C:smooth} implies that $\F$ is smooth. The description given in \cite{Touzet} (see \S \ref{S:oldresults})  implies that $\F$ is defined
by a closed holomorphic $1$-form with coefficients in a torsion line bundle. Therefore, to conclude the proof of Theorem \ref{T:F}, it remains
to show that the closed $1$-forms given by Theorem \ref{T:modulop} have coefficients in a torsion line bundle.

\subsubsection{Flat implies torsion }\label{SS:flator}
Let $\F$ be a non-uniruled foliation with $c_1(\KF)=0$,  given by closed rational $1$-form $\omega$ without zeroes divisor, and  with coefficients in a flat line bundle $L$.
Assume $\sing(\F) \neq \emptyset$, and write ${(\omega)}_\infty=\sum \lambda_D D$ as a sum of irreducible divisors with positive integers coefficients.

If all the residues of $\omega$ are zero, then, according to Corollary \ref{C:closed without residues}, we get that $H^1(X,N^* \F) \neq 0$, and we conclude that $\F$ is smooth (Corollary \ref{C:smooth}).
From now on, we will assume the existence of an irreducible component $D$ in the support of $(\omega)_{\infty}$ with non-zero residue. Corollary \ref{C:Frat} implies that $D$
dominates  $R_X$ through the maximal rationally connected meromorphic fibration $R: X \dashrightarrow R_X$.

Recall that rationally connected manifolds are  simply-connected, and consequently $L$ is trivial in these
manifolds.  Thus,  we have only to deal with $X$ uniruled,  with rational quotient $R_X$ not reduced to a point.
Let us denote by $U\subset R_X$ the Zariski open subset such that the fibration $R$ over $U$ is a regular one. Let us pick a small  open ball $B$ on  $U$.
Over $R^{-1}(B)$ our flat line bundle $L$ is trivial since the fibers of $R$ are rationally-connected.
Therefore, we can represent $\omega$ in $R^{-1}(U)$ by a meromorphic $1$-form normalized in such a way that  its residue along a branch of $D$ is equal to $1$.
Since  there are only finitely many choices involved,  this enables us to conclude that $L$ is torsion. Therefore, after a finite \'etale covering, $\F$ is defined by
a closed rational $1$-form. This completes the  proof of Theorem \ref{T:F}.
\qed


\section{Structure}\label{S:structure2}

In this section, we will present a proof of our main result, namely Theorem \ref{TI:X}.

\subsection{Factoring out a manifold with trivial canonical bundle}\label{S:CY}
Let $\F$ be a codimension one foliation with numerically trivial canonical bundle 	and canonical singularities on a projective  manifold $X$.
Assume that the general leaf of $\F$ is not algebraic.
According to Theorem \ref{T:F}, we can (and will) assume that $\F$ is defined by a closed rational $1$-form $\omega$ without divisorial components in its zero set, after replacing $X$ by an \'etale covering.
Let $D$ be the  polar divisor of $\omega$. Consider the natural morphism
\[
	   \phi : T\F \to H^0(X, \Omega^1_X(\log D))^*  \otimes \mathcal O_X
\]
obtained by contracting germs of vector fields tangent to $\F$ with global logarithmic $1$-forms with poles on $D$. Here, we are exploiting the $\F$-invariance of  $D$ in order to get holomorphic $1$-forms along leaves from
logarithmic $1$-forms on $X$. Let $\H$ be the foliation with tangent sheaf equal to the kernel of $\phi$. Alternatively, $\H$ is defined by the kernel of $\omega$ intersected with the kernel
of all the elements of  $H^0(X, \Omega^1_X(\log D))$. From this description it is clear that $\det N \H \le  N \F = - \KX$.

\begin{remark}
If one does not impose restrictions on the poles of logarithmic $1$-forms, it may happen that they end up 
being non-closed, see the discussion at \cite[pp. 601-603]{BrunellaMendes}. For instance the $1$-form on $\mathbb P^2$ defined on affine coordinates by $(x^2 - y^5)^{-1}(2ydx -5xdy)$ has simple poles but is not closed. In our setup, where $D$ is normal crossing in codimension one, any $\omega \in H^0(X, \Omega^1_X(\log D))$ is automatically closed. To verify this fact, first observe that the residue of $\omega$ along the smooth locus of an irreducible component $H$ of $D$ is a well-defined holomorphic function. Since such irreducible component is a normal hypersurface, this holomorphic function extends through the singularities of $H$ and therefore must be constant by compactness of $H$. Thus the residues of $\omega$ are constant. It follows that $d\omega$ is a holomorphic $2$-form and we can apply \cite[Proposition 7.1]{CDQL} to conclude that $d\omega=0$.
\end{remark}

\begin{lemma}\label{L:factorCY}
    Perhaps after passing to an \'{e}tale covering, the manifold $X$ is the product of a
    projective manifold $Y$ with trivial canonical bundle, and another projective manifold $Z$;  the foliation
    $\H$ coincides with the fibers of the projection to $Z$; and $\F$ is the pull-back of a foliation on $Z$ via this projection.
\end{lemma}
\begin{proof}
    Let $\alb: X - D \to H^0(X, \Omega^1_X(\log D))^* / H_1(X-D, \mathbb Z)$ be the quasi-albanese map in the sense of \cite{Iitaka}. 
    Recall from \cite[Proposition 3]{Iitaka}  that $\alb$ extends to a rational map on the whole $X$ and, as such, defines a (perhaps singular)
    foliation on $X$ by algebraic leaves. By definition, the foliation $\mathcal H$ is the intersection of the foliation defined by
    the fibers of the quasi-albanese map with $\F$. Since any abelian representation of $\pi_1(X-D)$ becomes trivial when restricted to an arbitrary fiber $F$
    of $\alb$, we deduce that $\omega_{|F}$, the restriction of $\omega$ to  $F$, is an exact rational $1$-form. It follows that $\H$ is a  foliation by algebraic leaves.

    Corollary \ref{C:uniruled} implies that $\det N\H =  N\F$ (otherwise $\H$, and consequently $\F$ would be uniruled). Thus $\KH = \KF$   and   Theorem \ref{T:fibration}
    implies the product structure for $X$ compatible with $\H$. The result follows.
\end{proof}

\subsection{Finding the abelian Lie group} We keep the notation used in  Section \ref{S:CY} above.

\begin{lemma}\label{L:Free}
	If $\F$ is defined by a closed
    rational $1$-form $\omega$ without codimension one zeros, does not admit a rational first integral,
    and the foliation $\H$  is the foliation by points, then the
	tangent sheaf of $\F$ is free, i.e., $T \F = \mathcal O_X^{\oplus \dim(\F)}$. Furthermore,
    the projective manifold $X$ is an equivariant compactification of an abelian Lie group $G$, and
    the foliation $\F$ is induced by a codimension one subalgebra of the Lie algebra of $G$.
\end{lemma}
\begin{proof} 
    If $\H$  is the foliation by points, then  $\phi$ is injective. Therefore, the dual of $\phi$
    \[
	     \phi' : H^0(X, \Omega^1_X(\log(D)) \otimes \mathcal O_X\to T^* \F
	\]
	must be surjective in codimension one.
    From the triviality of the  determinants of the source and the target, we deduce that $\ker \phi'$ is isomorphic to a certain number of copies of $\mathcal O_X$. Since $\F$ has no rational first integral, and logarithmic $1$-forms are closed, the kernel of $\phi'$ is either trivial  or of rank  one. In the former case $D \neq D_{red}$, while in the later case the kernel is generated by $\omega$, a logarithmic $1$-form defining $\F$. In both cases, we have that $T  \F$ is isomorphic to $\mathcal O_X^{\oplus \dim \F}$ and that $\phi'$ is everywhere surjective.   Notice that we have also proved  that
    \begin{equation}\label{E:dimh0}
        \dim_\C H^0(X, \Omega^1_X(\log D)) = \left\{\begin{array}{ll}
        \dim X & \text{ when } \omega \text{ is logarithmic, and} \\
        \dim X -1 & \text{ otherwise.}
    \end{array} \right.
    \end{equation}
    Hence, the $\mathbb C$-vector space $V$ of closed rational $1$-forms generated by $\omega$ and $H^0(X, \Omega^1_X(\log D))$ always  has the same dimension as $X$. Furthermore, this complex vector space
    generates the $\mathbb C(X)$-vector space of rational $1$-forms on $X$.

    Since $T \F$ is trivial, we can choose holomorphic vector fields $v_1, \ldots, v_{n-1}$ which generate $T\F$. Notice that the contraction of any of these vector fields with any of the
    $1$-forms in $V$ is a constant, since they are all in the kernel of $\omega$ and $D$ is $\F$-invariant. Cartan's formula for the exterior derivative implies they generate an abelian Lie
    algebra of dimension $n-1$. Integration defines an abelian subgroup $H$ of the algebraic group $\Aut^0(X)$ of dimension $n-1$. Since the general leaf of $\F$ is not algebraic, we have
    that $H$ is not a Zariski closed subgroup of $\Aut(X)$.

    Let $\Aut^0(X)$ be the connected component of the identity of the  group of automorphisms of $X$.
    Since $X$ is projective by assumption, $\Aut^0(X)$ is an algebraic group. Let $G \subset \Aut^0(X)$ be the Zariski closure of $H$.
    Notice that $G$ is also abelian of dimension at least equal to $n=\dim X$. To verify that the dimension
    of $G$ is actually equal to $\dim X$, let's asssume (aiming at a contradiction) that we have at  least $n+1$ linearly independent vector fields in the Lie algebra
    of $G$ say $v_1, \ldots, v_n, v_{n+1}$.  Without loss of generality, we can assume that the first $n$ vector fields generate $TX$ over a Zariski open subset.
    Hence, we can write $v_{n+1} = \sum_{i=1}^n f_i v_i$ for suitable rational functions
    $f_i \in \mathbb C(X)$. Since $G$ is abelian, we have that $0=[v_j,v_{n+1}] = \sum_{i=1}^n v_j(f_i) v_i$ for every $j$. It follows that the functions $f_1, \ldots, f_n$
    are constant and that the vector field $v_{n+1}$ is a linear combination of the vector fields $v_1, \ldots, v_n$ contrary to our assumption. We conclude that $G$ has dimension $n$.

   Finally observe that the group $G$ itself can be identified with the locus where the wedge product $v_1 \wedge \cdots \wedge v_n$ does not vanish. We conclude that $X$ is an equivariant
   compactification of $G$.
\end{proof}

\begin{remark}
    We can go one step further and prove that $X$ is either a compactification of a quasi-abelian variety, or  an equivariant compactification of an extension of a
    quasi-abelian variety  by $(\mathbb C,+)$. Indeed, the group $G$ surjects onto an Abelian variety with fibers given by abelian algebraic linear groups. By a classical result
    of Rosenlicht, these abelian algebraic linear groups are isomorphic to $(\mathbb C^*, \cdot)^r \times (\mathbb C,+)^s$.  If $s=0$, then $G$ is a quasi-abelian variety;
    if instead $s=1$, then $G$ is an extension by $(\mathbb C,+)$ of a quasi-abelian variety. The case $s\ge 2$ is impossible, because the Lie algebra of $(\mathbb C,+)^s$
    would intersect $H$ giving rise to rational curves contained in the leaves of $\F$.
\end{remark}

\subsection{Proof of the main result (Theorem \ref{TI:X} of the Introduction)}
If $\F$ has non canonical singularities, then Theorem \ref{T:uniruled iff non psef} implies that $\F$ is uniruled. If instead $\F$  has canonical singularities,
then according to Theorem \ref{T:F}, after passing to a finite \'etale covering,  $\F$ is a smooth isotrivial fibration, or  $\F$ is defined by a closed  rational $1$-form without divisorial components in its zero set.
When $\F$ is a  fibration,  the result follows from Theorem \ref{T:alg leaves}. From now on assume that $\F$ is defined by a closed  rational $1$-form without divisorial components in its zero set, and is not a fibration.
We can apply Lemma \ref{L:factorCY} to show that $X$ is, up to an \'etale covering, a product of a projective manifold $Y$ with numerically trivial canonical bundle with a projective manifold $Z$. Furthermore, $\F$ is the pull-back under the natural
projection $Y \times Z \to Z$ of a foliation $\G$ on $Z$ satisfying the assumptions of Lemma \ref{L:Free}. Theorem \ref{TI:X} follows.
\qed


\bibliographystyle{amsplain}

\end{document}